\documentclass[12pt]{article}

\usepackage{graphicx}
\usepackage[dvips]{epsfig}
\usepackage{authblk}

\usepackage{color}

\newtheorem{thm}{Theorem}[section]
\newtheorem{lemma}[thm]{Lemma}
\newtheorem{cor}[thm]{Corollary}
\newtheorem{prop}[thm]{Proposition}

\newtheorem{rem}[thm]{Remark}

\newenvironment{proof}[1][Proof]{\textbf{#1.} }{\ \rule{0.5em}{0.5em}}

\usepackage{latexsym}
\usepackage{amsfonts}
\usepackage{amssymb}
\usepackage{amsmath}
\usepackage[english]{babel}

\def\C{\mathcal C}
\def\K{\mathcal K}

\def\D{\Delta_n}
\def\R{{\cal R}}
\def\L{\Lambda}
\def\S{\mathcal S}

\def\T{\Theta}
\def\G{\Gamma}
\def\O{\Omega}

\def\g{\gamma}

\def\s{\sigma}
\begin{document}
%\date{}
\title{Compact $n$-manifolds via $(n+1)$-colored graphs:\\ a new approach}

\bigskip

 \author[*]{Luigi GRASSELLI}
 \author[**]{Michele MULAZZANI}
\affil[*]{\small Dipartimento di Scienze e Metodi dell'Ingegneria, Universit\`a di Modena e Reggio Emilia (Italy)} \affil[**]{Dipartimento di Matematica and ARCES, Universit\`a di Bologna (Italy)}
\renewcommand\Authands{ and }

\maketitle

\begin{abstract}
We introduce a representation via $(n+1)$-colored graphs of compact $n$-manifolds with (possibly empty) boundary, which appears to be very convenient for computer aided study and tabulation. Our construction is a
generalization to arbitrary dimension of the one recently given by Cristofori and Mulazzani in dimension three, and it is dual to the one given by Pezzana in the seventies.
In this context we establish some results concerning the topology of the represented manifolds: suspension, fundamental groups, connected sums and moves between graphs representing the same manifold. Classification results of compact orientable $4$-manifolds representable by graphs up to six vertices are obtained, together with some properties of the G-degree of 5-colored graphs relating this approach to tensor models theory. 
\end{abstract}

\bigskip

\small{

\thanks{


\smallskip


}
}

\bigskip
\bigskip

\section{\hskip -0.7cm . Introduction}

The representation of closed $n$-manifolds by means of $(n+1)$-colored graphs has
been introduced in the seventies by M. Pezzana's research group in Modena (see \cite{[FGG]}). In this type  of representation, an $(n+1)$-colored graph (which is an $(n+1)$-regular multigraph with a proper $(n+1)$-edge-coloration) represents a closed $n$-manifold if certain conditions on its subgraphs are satisfied. 

The study of this kind of representation has yielded several results, especially with regard to the definition of combinatorial invariants and their relations with topological properties of the represented manifolds (see \cite{[CC$_2$]}, \cite{[Cr$_1$]}, \cite{[CGG]} and \cite{[G$_4$]}).

During the eighties S. Lins introduced a representation of closed 3-manifolds via 4-colored graphs, with an alternative construction which is dual to Pezzana's one (see \cite{[Li]}). The extension of this representation to $3$-manifolds with boundary has been performed in \cite{[CM]}, where a compact $3$-manifold with (possibly empty) boundary is associated to any $4$-colored graph, this correspondence being surjective on the class of such manifolds without spherical boundary components. 

As a consequence, an efficient computer aided catalogation/classification of $3$-manifolds with boundary, up to some value of the order of the representing graphs, can be performed by this tool. For example, the complete classification of orientable $3$-manifolds with toric boundary representable by graphs of order $\le 14$ is given in \cite{[CFMT]} and \cite{[CFMT2]}. 

In this paper we generalize the above construction to the whole family of colored graphs of arbitrary degree $n+1$, showing how they represent compact $n$-manifolds with (possibly empty) boundary. This opens the possibility to introduce an efficient algorithm for
computer aided tabulation of $n$-manifolds with boundary, for any $n\ge 3$.

The construction is described in Section~3, while Section~4 deals with graph suspensions and their connections with the topological/simplicial suspension of the represented spaces.
A set of moves connecting graphs representing the same manifold is given in Section~5. However, these moves are  not sufficient to ensure the equivalence of any two graphs representing the same manifold. 
In Section~6 we associate to any $(n+1)$--colored graph a group which is strictly related to the fundamental group of the associated space and, therefore, it is a convenient tool for its direct computation (in many cases the two groups are in fact isomorphic). In Section~7 we establish the relation between the connected sum of graphs and the (possibly boundary) connected sum of the represented manifolds.  
In Section~8 a classification of all orientable 4-manifolds representable by colored graphs of order $\le 6$ is presented. Certain properties related with tensor models theory (see for example \cite{[Gu]} and \cite{[GR]}) involving the G-degree in dimension four are also obtained. More generally, as pointed out in \cite{[CCDG]}, the strong interaction between random tensor models and the topology of $(n+1)$-colored graphs makes significant the theme of enumeration and classification of all quasi-manifolds (or compact manifolds with boundary) represented by graphs of a given G-degree and the arguments presented in this paper might be a useful tool for this purpose.

\section{\hskip -0.7cm . Basic notions}

Throghout this paper all spaces and maps are considered in the PL-category, unless explicitely stated.

For $n\ge 1$, an {\it $n$-pseudomanifold} is a simplicial complex $K$ such that: (i) any $h$-simplex is the face of at least one $n$-simplex, (ii) each $(n-1)$-simplex is the face of exactly two $n$-simplexes and (iii) every two $n$-simplexes can be connected by means of a sequence of alternating $n$- and $(n-1)$-simplexes, each simplex being incident with the next one in the sequence. The notion of pseudomanifold naturally transfers to the underlying polyhedron $\vert K\vert$ of $K$. A pseudomanifold $K$ is an $n$-manifold out of a (possibly empty) subcomplex $S_K$ of dimension $\le n-2$, composed by the {\it singular} simplexes (i.e., the simplexes whose links are not spheres). We refer to $S_K$ as the {\it singular complex} of $K$ and to $\vert S_K\vert$ as the {\it singular set} of $\vert K\vert$. 

A {\it quasi-manifold} is a pseudomanifold such that the star of any simplex verifies condition (iii) above (see \cite{[G]}). When $n\ge 2$, an $n$-pseudomanifold $K$ is a quasi-manifold if and only if the link of any 0-simplex of $K$ is an $(n-1)$-quasi-manifold. It is easy to prove that the singular complex of an $n$-dimensional quasi-manifold has dimension $\le n-3$. 

For $n\ge 2$, a {\it singular $n$-manifold} is a quasi-manifold such that the link of any 0-simplex is a closed connected $(n-1)$-manifold. It easily follows that the link of any $h$-simplex of a singular manifold, with $h>0$, is an $(n-h-1)$-sphere. So the singular set of a singular manifold is a (possibly empty) finite set of points, and this property caracterizes singular manifolds among quasi-manifolds.
Note that in dimension three (resp. dimension two) any quasi-manifold is a singular manifold (resp. is a closed surface).

A {\it pseudo-simplicial complex} is an $n$-dimensional ball complex in which every $h$-ball, considered with all its faces, is abstractly isomorphic to the $h$-simplex complex. It is a fact that the first barycentric subdivision of a pseudo-simplicial complex is an abstract simplicial complex (see  \cite{[HW]}). The notions of pseudomanifolds, quasi-manifolds and singular manifolds can be extended to the setting of pseudo-simplicial complexes by considering their barycentric subdivisions.
If $K$ is a (pseudo-)simplicial complex we will denote by $\vert K\vert$ its underlying space. 

Let $n$ be a positive integer and $\G=(V(\G),E(\G))$ be a finite graph which is $(n+1)$-regular (i.e., any vertex has degree $n+1$), possibly with multiple edges but without loops. An \textit{$(n+1)$-edge-coloration} of $\G$ is any map $\g:E(\G)\to C$, where $\vert C\vert=n+1$. The coloration is called \textit{proper} if adjacent edges have different colors. An edge $e$ of $\G$ such that $\gamma(e)=c$ is also called a \textit{$c$-edge}. Usually we set $C=\D=\{0,1,\ldots,n\}$.

An \textit{$(n+1)$-colored graph} is a connected $(n+1)$-regular graph equipped with a proper $(n+1)$-coloration on the edges. It is easy to see that any $(n+1)$-colored graph has even order and it is well known that any bipartite $(n+1)$-regular graph admits a proper $(n+1)$-coloration (see \cite{[FW]}).
%but not all (non-bipartite) $(n+1)$-regular graphs of even order can be endowed with a proper $(n+1)$-coloration. On the other side, 

If $\Delta\subset\D$, denote by $\G_{\Delta}$ the subgraph of $\G$ obtained by dropping out from $\G$ all $c$-edges, for any $c\in\widehat\Delta=\D-\Delta$. Each connected component $\L$ of $\G_{\Delta}$ is called a \textit{$\Delta$-residue} - as well as a \textit{$\vert\Delta\vert$-residue} - of $\G$, indicated by $\L\prec\G$. Of course, 0-residues are vertices, 1-residues are edges, 2-residues are bicolored cycles with an even number of edges, also called {\it bigons}. An $h$-residue of $\G$ is called {\it essential} if $2\le h\le n$. The number of $\Delta$-residues of $\G$ will be denoted\footnote{For $\Delta=\{i,j\}$ we use the simplified notation $g_{i,j}$ instead of $g_{\{i,j\}}$.} by $g_{\Delta}$. An $(n+1)$-colored graph $\G$ is called {\it supercontracted} \footnote{Such type of graphs were called {\it contracted} in \cite{[G]} and in related subsequent papers, but in \cite{[CM]} the term contracted refers to a more general class of colored graphs.} if $g_{\widehat c}=1$, for any $c\in\D$, where $\widehat c=\D-\{c\}$. If $\L$ and $\L'$ are residues of $\G$ and $\L'$ is a proper subgraph of $\L$, then $\L'$ is also a residue of the colored graph $\L$.

The set of all $h$-residues of an $(n+1)$-colored graph $\G$ is denoted by $\R_h(\G)$ and the set $$\R(\G)=\bigcup_{0\le h\le n}\R_h(\G)$$ results to be partially ordered by the relation $\preceq$ (as usual $\preceq$ means either $\prec$ or $=$) and will play a central role in our discussion.

%Moreover, for any essential $h$-residue $\L$ of $\G$ and for any $k<h$, define $\R_k(\L)=\{\L'\in\R_k(\G)\mid \L'\prec \L\}$. 

For elementary notions about graphs we refer to \cite{[Wh]} and
for general PL-topology we refer to \cite{[Gl]}, \cite{[HW]} and \cite{[Hu]}.

\section{\hskip -0.7cm . The construction}\label{construction}

\subsection{\hskip -0.7cm . The quasi-manifold $\widehat M_{\G}$}\label{quasi-manifold}

Given an $(n+1)$-colored graph $\G$, we associate to it an $n$-dimensional complex $\C_{\G}$, as well as its underlying space $\widehat M_{\G}=\vert \C_{\G}\vert$, obtained by attaching to $\G$ ``cone-like'' cells in one-to-one correspondence with the essential residues of $\G$, where the dimension of each cell is the number of colors of the associated residue. For the sake of conciseness, the $h$-skeleton of $\C_{\G}$ will be denoted by $\G^{(h)}$, for any $h=0,1,\ldots,n$. 

First of all we consider vertices and edges of $\G$ as 0-dimensional and 1-dimensional cells respectively, with the natural incidence structure. So the 0-skeleton $\G^{(0)}$ of $\C_{\G}$ is $V(\G)$ and the 1-skeleton $\G^{(1)}$ is the graph $\G$, considered (as well as its essential residues) as a 1-dimensional cellular complex in the usual way. Moreover define $\L^{(1)}=\L$, for any essential residue $\L$ of $\G$.

If $n=1$ then $\G$ is just a bigon and it has no essential residues. In this case $\C_{\G}=V(\G)\cup E(\G)$ and $\widehat M_{\G}=\G\cong S^1$.

If $n\ge 2$, we proceed by induction via a sequence of cone attachings $Y\to Y\cup C(X)$, where $X$ is a subspace of $Y$ and $C(X)=(X\times [0,1])/(x,1)\sim(x',1)$ is the cone over $X$ which is attached to $Y$ via the map $(x,0)\in C(X)\mapsto x\in Y$.\footnote{When we want to stress the presence of the cone vertex $V=(X\times\{1\})/\sim$ we will use the notation $C_V(X)$ instead of $C(X)$.} At each step $h=2,\ldots,n$ these attachings are in one-to-one correspondence with the elements of $\R_h(\G)$, according to the following algorithm.

By induction on $h=2,\ldots,n$ let $$\G^{(h)}=\G^{(h-1)}\bigcup_{\L\in \R_{h}(\G)}C(\L^{(h-1)})$$ and, for any essential $h'$-residue $\T$ of $\G$ with $h'> h$, define $\T^{(h)}=\T^{(h-1)}\cup_{\L\in \R_{h}(\T)}C(\L^{(h-1)})$, which is obviously a subspace of $\G^{(h)}$. 

The final result of this process, namely $\G^{(n)}$, is the space $\widehat M_{\G}$, and we say that $\G$ {\it represents} $\widehat M_{\G}$.

For any $h=2,\ldots,n$ the $h$-cells of $\C_{\G}$ are the cones $c_{\L}=C(
\L^{(h-1)})$, for any $\L\in\R_h(\G)$, and the vertex of the cone is denoted by $V_{\L}$. 
Moreover, we can consider any edge $e\in E(\G)$ as the result of a cone on its endpoints with vertex $V_e$, as well as any $v\in V(\G)$ can be considered as the result of a cone on the empty set with vertex $V_v=v$. As a consequence, each cell of $\C_{\G}$ is associated to a residue of $\G$ and it is a cone over the union of suitable cells of lower dimension. The set $\C_{\G}=\{c_{\L}\mid \L\in\R(\G)\}$ is called the {\it cone-complex associated to $\G$} and we have $$\widehat M_{\G}=\vert \C_{\G}\vert=\bigcup_{\L\in\R(\G)}c_{\L}=\bigcup_{\L\in\R_n(\G)}c_{\L}.$$ 
For any $\L\in\R_h(\G)$ the space $\widehat M_{\L}=\vert \C_{\L}\vert$ is an $(h-1)$-dimensional subspace of $\widehat M_{\G}$. In particular, $\widehat M_{\L}=S^0$ if $\L$ is a 1-residue and $\widehat M_{\L}=S^{-1}=\emptyset$ if $\L$ is a 0-residue. Observe that, with this notation, $c_{\L}=C(\widehat M_{\L})$. In the following $\bar c_{\L}$ will denote the cone-complex $\C_{\L}\cup\{c_{\L}\}$.

The set $\Upsilon=\{V_{\L}\mid \L\in\R(\G)\}$ of the cone-vertices is a 0-dimensional subspace of $\widehat M_{\G}$. A cell $c_{\L'}$ is a proper face of a cell $c_{\L}$ (written as usual $c_{\L'}<c_{\L}$) when $c_{\L'}\subset c_{\L}$. Therefore, $c_{\L'}< c_{\L}$ if and only if $\L'\prec\L$. 

It is worth noting that the cells of $\C_{\G}$ are not in general balls. In fact, an $h$-cell $c_{\L}$ is a ball if and only if $\widehat M_{\L}$ is an $(h-1)$-sphere. 

An $h$-residue $\L$ of $\G$ is called {\it ordinary} if $\widehat M_{\L}$ is an $(h-1)$-sphere, otherwise it is called {\it singular}.
Of course, all 0-, 1- and 2-residues are ordinary. As proved later (see Corollary~\ref{ferri}), $\widehat M_{\G}$ is a closed $n$-manifold if and only if all residues of $\G$ are ordinary and in this case the cone-complex $\C_{\G}$ results to be a genuine (regular) CW-complex. 

If $n=2$ the above construction just reduces to the attaching of a disk along its boundary to any bigon of $\G$, and therefore $\widehat M_{\G}$ is a closed surface. 

If $n=3$ the construction, which was introduced for closed 3-manifolds in \cite{[LM]}, has an additional step consisting in performing the cone over any 3-residue of $\G$, considered together with the disks previously attached to its bigons. As shown in \cite{[LM]}, $\widehat M_{\G}$ is a closed 3-manifold if and only if all 3-residues of $\G$ are ordinary.\footnote{The condition is equivalent to the arithmetic one: $g_0+g_3=g_2$, where $g_0,g_3$ and $g_2$ are respectively the number of vertices, 3-residues and bigons of $\G$ (see \cite{[LM]}).} 
On the contrary, if some 3-residue is singular then $\widehat M_{\G}$ is a 3-dimensional singular manifold whose singularities are the cone points of the cells corresponding to the singular 3-residues. Note that it is easy to check whether a 3-residue $\L$ is ordinary or not by Euler characteristic arguments: if $v$ and $b$ are the number of vertices and bigons of $\L$, then $\L$ is ordinary if and only if $b-v/2=2$. 

The cone-complex $\C_{\G}$ admits a natural ``barycentric'' subdivision $\C'_{\G}$, which results to be a simplicial complex, as follows.
The 0-simplexes of $\C'_{\G}$ are the cone vertices (i.e., the elements of $\Upsilon$), so they are in one-to-one correspondence with the elements of $\R(\G)$. 
The set of $h$-simplexes of $\C'_{\G}$ is in one-to-one correspondence with the sequences $(\L_0,\L_1,\ldots,\L_h)$ of residues of $\G$, such that $\L_0\prec\L_1\prec\cdots\prec\L_h$. Namely, the $h$-simplex $\s^h$ of $\C'_{\G}$ associated to $(\L_0,\L_1,\ldots,\L_h)$ has vertices $V_{\L_0},V_{\L_1},\ldots, V_{\L_h}$ and it is defined by applying to $V_{\L_0}$ the sequence of cone constructions corresponding to the residues $\L_1,\ldots,\L_h$, in this order. Therefore $$\s^h=\langle V_{\L_0},V_{\L_1},\ldots, V_{\L_h}\rangle=C_{V_{\L_h}}(C_{V_{\L_{h-1}}}(\cdots(C_{V_{\L_1}}(V_{\L_0}))\cdots)).$$ It is interesting to note that $\C'_{\G}$ is isomorphic to the order complex of the poset $\R(\G)$ (see Section 9 of \cite{[Bjo]}).
In the following with the notation $\langle V_{\L_0},V_{\L_1},\ldots, V_{\L_h}\rangle$ we always mean that $\L_0\prec\L_1\prec\cdots\prec\L_h$.

If $\L$ is an $h$-residue of $\G$, we denote by $c'_{\L}$ the subcomplex of $\C'_{\G}$ obtained by restricting the barycentric subdivision to the cell $c_{\L}$. Therefore a $k$-simplex $\s^k=\langle V_{\L_0},V_{\L_1},\ldots, V_{\L_k}\rangle$ is a simplex of $c'_{\L}$ if and only if $\L_k\preceq \L$. It is a standard fact that $c'_{\L}=V_{\L}\star\dot c'_{\L}$, where \hbox{$\dot c'_{\L}=\{\langle V_{\L_0},V_{\L_1},\ldots, V_{\L_k}\rangle\in c'_{\L}\mid \L_k\prec\L\}$.} Note that $\vert\dot c'_{\L}\vert=\widehat M_{\L}$.

For $n\ge 0$, we denote by $\bar s_n$ the complex composed by the standard $n$-simplex $s_n$ and all its faces and by $\dot s_n$ its boundary complex. Therefore $\bar s'_n$ and $\dot s'_n$ denote their barycentric subdivisions, respectively. Observe that $\vert\dot s_n\vert=\vert\dot s'_n\vert=S^{n-1}$, for any $n\ge 0$. In the following we set, as usual,  $A\star\emptyset=A$, for any simplicial complex $A$.

\begin{prop} \label{link} Let $\s^h=\langle V_{\L_0},V_{\L_1},\ldots, V_{\L_h}\rangle$ be an $h$-simplex of $\C'_{\G}$. If $\L_i$ is a $d_i$-residue of $\G$, for $i=0,1,\ldots,h$, then $\textrm{Link}\,(\s^h,\C'_{\G})$ is isomorphic to the complex $\dot c'_{\L_0}\star \dot s'_{d_1-d_0-1}\star\cdots\star\dot s'_{d_h-d_{h-1}-1}\star \dot s'_{n-d_h-1}$. Hence, $\vert\textrm{Link}\,(\s^h,\C'_{\G})\vert$ is homeomorphic to $\widehat M_{\L_0}\star S^{n-h-d_0-1}$.
\end{prop}

\begin{proof} For $i=1,\ldots,h$ suppose $\L_i$ is a $D_i$-residue of $\G$, with $|D_i|=d_i$, and set $D_{h+1}=\D$.
A simplex $\tau^s=\langle V_{\O_0},V_{\O_1},\ldots, V_{\O_s}\rangle$ belongs to $\textrm{Link}\,(\s^h,\C'_{\G})$ if there exist $j_0,j_1,\ldots,j_h$ with  $0\le j_0\le j_1\le\ldots\le j_h< n$, such that $\O_0\prec \ldots\prec \O_{j_0}\prec \L_0$, $\L_0\prec\O_{j_0+1}\prec \ldots\prec \O_{j_1}\prec \L_1$, $\ldots$, $\L_{h-1}\prec\O_{j_{h-1}+1}\prec \ldots\prec \O_{j_h}\prec \L_h$, $\L_{h}\prec\O_{j_{h}+1}\prec \ldots\prec \O_{s}$. Of course, the simplex $\tau_0=\langle V_{\O_0},\ldots, V_{\O_{j_0}}\rangle$ is a generic element of $\dot c'_{\L_0}$. 
On the other hand, for $i=1,\ldots,h+1$ the chain $\O_{j_{i-1}+1}\prec\ldots\prec \O_{j_i}$ is a generic element of the order complex of the subposet $R_i$ of $\R(\G)$ defined by $R_i=\{\O\in\R(\G)\mid\L_{i-1}\prec\O\prec\L_{i}\}$. It is easy to see that the poset $R_i$ is isomorphic to the poset of the proper subsets of $D_i-D_{i-1}$, via the correspondence $\O\mapsto D'-D_{i-1}$, where $D'$ is the set of colors of $\O$. 
Since the order complex of the proper subsets of a finite set $X$ is isomorphic to $\dot s'_{\vert X\vert-1}$, we obtain the proof.
\end{proof}

\begin{cor} \label{ferri} Let $\G$ be an $(n+1)$-colored graph, then $\widehat M_{\G}$ is a closed manifold if and only if all $n$-residues of $\G$ are ordinary.
\end{cor}

\begin{proof} Since $\widehat M_{\G}$ is a closed manifold if and only if the link of any 0-simplex of $\C'_{\G}$ is an $(n-1)$-sphere, the result immediately follows from Proposition~\ref{link}.\footnote{Notice that a proof of Corollary~\ref{ferri} is given in \cite{[Fe]} using the dual construction described later on.}
\end{proof}

\medskip

As a consequence, all residues of a ordinary residue are ordinary.

The singular complex of $\C'_{\G}$ is denoted by $\S_{\G}$, and therefore $\vert\S_{\G}\vert$ indicates the singular set of $\widehat M_{\G}$. By Proposition~\ref{link}, an $h$-simplex $\s^h=\langle V_{\L_0},V_{\L_1},\ldots,V_{\L_h}\rangle$ of $\C'_{\G}$ belongs to $\S_{\G}$ if and only if $\L_0$ is a singular residue of $\G$ (and therefore any $\L_i$ is a singular residue, for $i=1,\ldots,h$). As mentioned before, the set $\vert\S_{\G}\vert$ is empty when $n=2$ and finite when $n=3$.

The set of ordinary (resp. singular) residues of $\G$ will be denoted by $\R'(\G)$ (resp. $\R''(\G)$) and let $\R'_h(\G)=\R_h(\G)\cap \R'(\G)$ (resp. $\R''_h(\G)=\R_h(\G)\cap \R''(\G)$), for any $0\le h\le n$. Of course, $\R''_h(\G)=\emptyset$ for all $h\le 2$. 
If $S$ is a connected component of $\vert \S_{\G}\vert$ then define $R_S\subseteq \R''_n(\G)$ by setting $R_S=\{\L\in\R''_n(\G)\mid V_{\L}\in S\}$. As a consequence, $S$ is a single point if and only if $\vert R_S\vert=1$. It is easy to see that $\dim(S)\le\vert R_S\vert-1$.

\begin{lemma} \label{full} Let $\G$ be an $(n+1)$-colored graph. If $\S_{\G}$ is non-empty then it is a full subcomplex of $\C'_{\G}$ and $\dim(\S_{\G})=n-\min\{h\mid \R''_h(\G)\neq\emptyset\}\le n-3.$
\end{lemma}

\begin{proof} Let $\s^k=\langle V_{\L_0},V_{\L_1},\ldots,V_{\L_k}\rangle$ be a simplex of $\C'_{\G}$ with all vertices belonging to $\S_{\G}$. As a consequence, $\L_i$ is a singular residue for all $i=0,1,\ldots,k$ and therefore $\s^k$ is a simplex of $\S_{\G}$. This prove that $\S_{\G}$ is a full subcomplex of $\C'_{\G}$. Since all 2-residues are ordinary, any simplex of $\S_{\G}$ has dimension $<n-2$. Moreover, if $\L$ is a singular $n$-residue, and $(\L_0=\L,\L_1,\ldots,\L_{n-h})$ is a maximal chain in $\R(\G)$ then $\langle V_{\L_0},V_{\L_1},\ldots,V_{\L_{n-h}}\rangle$ is an $(n-h)$-simplex of $\S_{\G}$. This concludes the proof.
\end{proof}

\medskip

The above construction of $\widehat M_{\G}$ is dual to the one given by Pezzana in the seventies (see for example the survey paper \cite{[FGG]}), which associates an $n$-dimensional pseudo-simplicial complex $\K_{\G}$ to the $(n+1)$-colored graph $\G$ in the following way:

\begin{itemize}
\item[•] (i) take an $n$-simplex $s_v$ for each $v\in V(\G)$ and color its vertices injectively by $\D$;
\item[•] (ii) if $v,w\in V(\G)$ are joined by a $c$-edge of $\G$, glue the $(n-1)$-faces of $s_{v}$ and $s_{w}$ opposite to the vertices colored by $c$, in such a way that equally colored vertices are identified together.
\end{itemize}

As a consequence, $\K_{\G}$ inherits a coloration on its vertices, thus becoming a balanced\footnote{An $n$-dimensional pseudo-simplicial complex is called {\it balanced} if its vertices are labelled by a set of $n+1$ colors in such a way that any 1-simplex has vertices labelled with different colors.} pseudo-simplicial complex.
This construction yields a one-to-one inclusion reversing correspondence $\L\leftrightarrow s_{\L}$ between the residues of $\G$ and the simplexes of $\K_{\G}$, in such a way that the $h$-simplex $s_{\L}$ of $\K_{\G}$ having vertices colored by $\Delta\subset\D$ is associated to the $\widehat \Delta$-residue of $\G$ having vertices corresponding via (i) to the $n$-simplexes of $\K_{\G}$ containing $s_{\L}$. 
%It is important to note that the first barycentric subdivision $\K'_{\G}$ of $\K_{\G}$ is a simplicial complex.

\begin{prop} The simplicial complexes $\C'_{\G}$ and $\K'_{\G}$ are isomorphic. Therefore \hbox{$\vert \K_{\G}\vert\cong\vert \C_{\G}\vert=\widehat M_{\G}$.} 
\end{prop}

\begin{proof} If we denote by $\widehat s$ the barycenter of the simplex $s$ of $\K_{\G}$, then the bijective map $d: S^0(\K'_{\G})\to S^0(\C'_{\G})$ between the 0-skeletons of the two complexes defined by $d(\widehat{s_{\L}})=V_{\L}$, for any $\L\in\R(\G)$, induces an isomorphism between the simplicial complexes $\K'_{\G}$ and $\C'_{\G}$.
\end{proof}

\medskip

\begin{rem}
The duality between the complexes $\C_{\G}$ and $\K_{\G}$ is given by the fact that $\C_{\G}$ is the complex $\K_{\G}^*$ dual to $\K_{\G}$ (i.e., composed by the dual cells of the simplexes of $\K_{\G}$). The definition of dual complex is analogous to the one in the simplicial case (see for example page 29 of \cite{[Hu]}), and it is well defined also in the pseudo-simplicial case since the barycentric subdivision $\K'_{\G}$ of $\K_{\G}$ is a simplicial complex: if $s_{\L}$ is an $h$-simplex of $\K_{\G}$ with vertex set $V$, then its {\it dual cell} $s_{\L}^*$ is the $(n-h)$-dimensional subcomplex of $\K'_{\G}$ given by
$s_{\L}^*=\cap_{v\in V}{\rm star}(v,\K'_{\G}).$ Therefore the cone-complex is exactly the complex $\C_{\G}=\{|s_{\L}^*|\mid s_{\L}\in \K_{\G}\}.$ 
%\footnote{Note that $\vert s^*_{\L}\vert=c_{\L}$ and $\vert s_{\L}\vert \cap \vert s^*_{\L}\vert=\{V_{\L}\}$, for any $\L\in\R(\G)$.} 
Moreover, the singular set of $\widehat M_{\G}$ is also the underlying space of the subcomplex $\bar \S_{\G}$ of $\K_{\G}$ composed by the simplexes $s_{\L}$ such that $\L$ is singular. It is easy to see that $\S_{\G}=\bar\S'_{\G}$. Note that maximal simplexes of $\bar \S_{\G}$ correspond to minimal singular residues of $\G$ (i.e., singular residues having no singular subresidues). 
\end{rem}

It is not difficult to see that $\C'_{\G}=\K'_{\G}$ is an $n$-dimensional quasi-manifold. Viceversa, any $n$-dimensional quasi-manifold admits a representation by $(n+1)$-colored graphs:

\begin{prop} \cite{[G]} Let $K$ be an $n$-dimensional (pseudo-)simplicial complex. Then there exists an $(n+1)$-colored graph $\G$ such that $\widehat M_{\G}\cong \vert K\vert$ if and only if $K$ is a quasi-manifold. 
\end{prop}

\medskip

The following result is straightforward:

\begin{lemma} \label{singular manifold} Let $\G$ be an $(n+1)$-colored graph, then $\widehat M_{\G}$ is a singular manifold if and only if $\R''_h(\G)=\emptyset$ for any $h<n$.
\end{lemma}

\begin{proof} If $\widehat M_{\G}$ is a singular manifold then $\dim(\S_{\G})\le 0$ and therefore, by Lemma~\ref{full}, $\R''_h(\G)=\emptyset$ for any $h<n$. On the contrary,  if $\R''_h(\G)=\emptyset$ for any $h<n$ then $\widehat M_{\G}$ is a singular manifold by Proposition~\ref{link}, since $\widehat M_{\L_0}$ is a closed manifold when $\L_0$ is an $n$-residue of $\G$.
\end{proof}

\subsection{\hskip -0.7cm . The manifold $M_{\G}$}\label{manifold}

As previously noticed, $\widehat M_{\G}$ is not always a manifold since it may contain singular points. In order to obtain a manifold we remove the interior of a regular neighborhood of the singular set $|\S_{\G}|$ in $\widehat M_{\G}$.

\begin{lemma} Let $N(\S_{\G})$ be a regular neighborhood of $\vert\S_{\G}\vert$ in $\widehat M_{\G}$, then $\widehat M_{\G}-\textrm{int}(N(\S_{\G}))$ is a compact $n$-manifold with (possibly empty) boundary.
\end{lemma}

\begin{proof} If $\S_{\G}=\emptyset$ the result is trivial since $\widehat M_{\G}$ is a closed $n$-manifold. If $\S_{\G}\neq\emptyset$ then Theorem~5.3 of \cite{[Co1]}, used with $Y=\emptyset$, says that the topological boundary $\textrm{bd}(N(\S_{\G}))$ of $N(\S_{\G})$ in $\widehat M_{\G}$ is bicollared in the (non-compact) $n$-manifold $\widehat M_{\G}-\vert\S_{\G}\vert$. It immediately follows that $M_{\G}=\widehat M_{\G}-\textrm{int}(N(\S_{\G}))$ is a compact $n$-manifold with boundary $\partial M_{\G}=\textrm{bd}(N(\S_{\G}))$.
\end{proof}

\medskip

We define $M_{\G}=\widehat M_{\G}-\textrm{int}(N(\S_{\G}))$ and say that the $(n+1)$-colored graph $\G$ also {\it represents} $M_{\G}$. If $\S_{\G}$ is empty then $M_{\G}=\widehat M_{\G}$ is a closed $n$-manifold. Otherwise, by the previous lemma $M_{\G}$ is a compact $n$-manifold with non-empty boundary. 

The next proposition gives a combinatorial description of $M_{\G}$. Recall that, for a subcomplex $H$ of a simplicial complex $K$, the {\it simplicial neighborhood of $H$ in $K$} is the subcomplex $N(H,K)$ of $K$ containing all simplexes of $K$ not disjoint from $H$, and their faces. Moreover, the {\it complement of $H$ in $K$} is the subcomplex $C(H,K)$ of $K$ containing all simplexes of $K$ disjoint from $H$ and set $\dot N(H,K)=N(H,K)\cap C(H,K)$.

\begin{prop} \label{manifold2} If $\G$ is an $(n+1)$-colored graph then $$M_{\G}=\vert C(\S'_{\G},\C''_{\G})\vert=\bigcup_{\L\in \R'(\G)}|\textrm{Star}(V_{\L},\C''_{\G})|.$$ Moreover, $\partial M_{\G}=\vert\dot N(\S'_{\G},\C''_{\G})\vert$.
\end{prop}

\begin{proof} Since $\S_{\G}$ is a full subcomplex of $\C'_{\G}$, we can choose $N(\S_{\G})$ in the first barycentric subdivision $\C''_{\G}$ of the complex $\C'_{\G}$, by defining $N(\S_{\G})=\vert N(\S'_{\G},\C''_{\G})\vert$ (see \cite{[Co1]}), and therefore we have $M_{\G}=\vert C(\S'_{\G},\C''_{\G})\vert$ and $\partial M_{\G}=\vert\dot N(\S'_{\G},\C''_{\G})\vert$. It is easy to realize that $N(\S'_{\G},\C''_{\G})=\cup_{v\in S^0(\S_{\G})}\textrm{Star}(v,\C''_{\G})=\cup_{\L\in \R''(\G)}\textrm{Star}(V_{\L},\C''_{\G})$ and $C(\S'_{\G},\C''_{\G})=\cup_{v\in (S^0(\C'_{\G})-S^0(\S_{\G}))}\textrm{Star}(v,\C''_{\G})=\cup_{\L\in \R'(\G)}\textrm{Star}(V_{\L},\C''_{\G})$.
\end{proof}

\medskip

As an immediate consequence of the fact that $\widehat M_{\G}$ is a quasi-manifold it follows that $M_{\G}$ is connected. Moreover, the connected components of $\partial M_{\G}$ are in one-to-one correpondence with the connected components of $\vert\S_{\G}\vert$, since the simplexes of $\S'_{\G}$ have connected links in $\C''_{\G}$.

\begin{rem} \label{cylinder}
The compact $n$-manifold $M_{\G}$ can be obtained from $\G$ by an alternative algorithm, which differs from the one producing $\widehat M_{\G}$ only in correspondence of singular residues, where the cone constructions are replaced by cylinder ones. Namely, for any singular $h$-residue $\L$, instead of attaching to $\G^{(h-1)}$ the cone $C(\L^{(h-1)})$ along the base, we attach the cylinder $\textrm{Cyl}(\L^{(h-1)})$ along one of the two bases. 
In order to prove that, it suffices to show that $M_{\G}\cap c_{\L}=M_{\L}\times I$, for any singular residue $\L$ of $\G$. This can be achieved by applying Lemma~1.22 of \cite{[Hu]}, with $V=V_{\L}$, $K=C(\S'_{\L},\C''_{\L})$ and $L=(V_{\L}\star D)\cup C((V_{\L}\star\S_{\L})',(V_{\L}\star\C'_{\L})')$, where $D=C((V_{\L}\star\S_{\L})',(V_{\L}\star\C'_{\L})')\cap\textrm{Star}(V_{\L},(V_{\L}\star\C'_{\L})')$.
%By (ii), if $\widehat M_{\G}$ is not a sphere then $S_{K_1}= \{N\}\star\S_{\G}$, which is a full subcomplex of $K_1$ containing $N$. As a consequence, $N(S_{K_1})=\vert N(S'_{K_1},K_1')\vert$ and $M_1=\vert C(S'_{K_1},K_1')\vert=\vert\cup_{v\in(S^0(K_1)-S^0(S_{K_1}))}\textrm{Star}(v,K_1')\vert=\vert\cup_{v\in(S^0(\C'_{\G})-S^0(\S_{\G}))}\textrm{Star}(v,K_1')\vert$. The complex $D=C(S'_{K_1},K_1')\cap\textrm{Star}(N,K_1')$ is isomorphic to $C(\S_{\G},\C''_{\G})$ and $L=(N\star D)\cup C(S'_{K_1},K_1')$ is a subdivision of $N\star C(\S'_{\G},\C''_{\G})$ and $\textrm{Star}(N,L)\cap C(\S'_{\G},\C''_{\G})=\emptyset$. Then Lemma~1.22 of \cite{[Hu]} applies and $M_1=\vert C(S'_{K_1},K_1')\vert=\vert N\star C(\S'_{\G},\C''_{\G})\vert-\textrm{int}\vert\textrm{Star}(v,L)\vert$ is homeomorphic to $\vert C(\S'_{\G},\C''_{\G})\vert\times I=M_{\G}\times I$.
\end{rem}

\medskip

%Therefore, the boundary components of $M_{\G}$ are associated to a suitable set of $n$-residues of $\G$. When $\widehat M_{\G}$ is a singular manifold there are natural one-to-one correspondences between points of the finite set $|\S_{\G}|$, elements of $\R''_n(\G)$ and connected components of $\partial M_{\G}$.

All interesting compact connected $n$-manifolds can be represented by $(n+1)$-colored graphs, as stated by the next proposition.

\begin{prop} \label{existence} If $M$ is a compact connected $n$-manifold with a (possibly empty) boundary without spherical components, then there exists an $(n+1)$-colored graph $\G$ such that $M=M_{\G}$.
\end{prop}

\begin{proof} Let $\widehat M$ be the space obtained from $M$ by performing a cone over any component of $\partial M$, then $\widehat M$ is an $n$-dimensional singular manifold. By Theorem~1 of \cite{[CCG]}, there exists an $(n+1)$-colored graph $\G$ such that $\widehat M=\widehat M_{\G}$. 
The singular set $|\S_{\G}|$ is the set of vertices of the cones, and the union of the cones is a regular neighborhood of $|\S_{\G}|$ in $\widehat M_{\G}$. As a consequence, $M_{\G}=M$.
\end{proof}

\medskip

Since two compact $n$-manifolds are homeomorphic if and only if (i) they have the same number of spherical boundary components and (ii) they are homeomorphic after capping off by balls these components, there is no loss of generality in studying compact $n$-manifolds without spherical boundary components.

When $\widehat M_{\G}$ is a singular manifold, the boundary of $M_{\G}$ admits a simple characterization in terms of the spaces represented by the singular residues.

\begin{lemma} \label{boundary0} Let $\G$ be an $(n+1)$-colored graph such that $\widehat M_{\G}$ is a singular manifold and let $\L$ be an $n$-singular residue of $\G$. Then the component of $\partial M_{\G}$ corresponding to $\L$ is homeomorphic to $M_{\L}=\widehat M_{\L}$.
\end{lemma} 

\begin{proof} The component of $\partial M_{\G}$ corresponding to $\L$ is $B=\vert \textrm{Link}(V_{\L},\C''_{\G})\vert=\vert \textrm{Link} (V_{\L},(V_{\L}\star\C'_{\L})')\vert$.
An $h$-simplex of $(V_{\L}\star\C'_{\L})'$ is $\s^h=\langle \widehat L_0,\widehat L_1, \ldots\widehat L_h\rangle$, where $L_i$ is a chain $\L_{i,0}\prec\L_{i,1}\prec\cdots\prec\L_{i,s_i}$ in $\R(\L)\cup\{\L\}$, for $i=0,1,\ldots,h$, such that $L_0<L_1<\cdots<L_h$ and $\widehat L_i$ is the barycenter of the simplex $\langle V_{\L_{i,0}},V_{\L_{i,1}},\ldots,V_{\L_{i,s_i}}\rangle$. 
The simplex $\s^h$ belongs to $\textrm{Link} (V_{\L},(V_{\L}\star\C'_{\L})')$ (resp. to $\C''_{\L}$) if and only if $\L\neq\L_{0,0}$ is the last element of the chain $L_0$ (resp. $\L$ is not an element of the chain $L_h$). Then the map $\iota:S^0(\C''_{\L})\to S^0(\textrm{Link} (V_{\L},(V_{\L}\star\C'_{\L})'))$ defined by $\iota(\widehat L)=\widehat{L'}$, where $L$ is the chain $\L_0\prec\L_1\prec\cdots\prec\L_k$, with $\L_k\prec\L$, and $L'$ is the chain $\L_0\prec\L_1\prec\cdots\prec\L_k\prec\L$, induces an isomorphism between $\C''_{\L}$ and $\textrm{Link} (V_{\L},(V_{\L}\star\C'_{\L})')$. As a consequence, $B$ is homeomorphic to $\vert\C''_{\L}\vert=\widehat M_{\L}$ ($= M_{\L}$ since all residues of $\L$ are ordinary).
\end{proof}

\begin{cor} If $\widehat M_{\G}$ is a singular manifold then $\partial M_{\G}$ has no spherical components.
\end{cor}

%\medskip

If $\dim(\S_{\G})=1$, the graph $\G$ has no singular $(n-2)$ residues, and any singular $(n-1)$-residue $\O$ is contained in exactly two $n$-residues $\L$ and $\L'$. 
Using the isomorphism $\iota$ defined in the proof of Lemma~\ref{boundary0}, we can suppose that $\widehat M_{\O}=M_{\O}$ is a boundary component of both $M_{\L}$ and $M_{\L'}$. Therefore, we can define the space $M_{\L}\cup^{\partial} M_{\L'}$ by gluing $M_{\L}$ with $M_{\L'}$ along their common boundary components (corresponding to common singular $(n-1)$-residues). Using this trick the boundary of $M_{\G}$ can be described as gluings of the manifolds with boundary corresponding to the singular $n$-residues of $\G$. 

\begin{prop} \label{boundary1} Let $\G$ be an $(n+1)$-colored graph such that $\dim(\S_{\G})=1$ and let $S$ be a connected component of $\vert\S_{\G}\vert$. Then the component of $\partial M_{\G}$ corresponding to $S$ is homeomorphic to $\bigcup^{\partial}_{1\le i\le s}M_{\L_i}$, where $\L_1,\ldots,\L_s$ are the $n$-residues of $R_S$.
\end{prop} 

\begin{proof} 
Let $\L$ be a singular $n$-residue of $\G$ belonging to $R_S$, and let $\O_1,\O_2,\ldots,\O_m$ be the singular $(n-1)$-residues of $\L$. It suffices to prove that (i) $\vert\dot N(\S'_{\G},\C''_{\G})\cap\bar c''_{\L}\vert$ is homeomorphic to $M_{\L}$ and (ii) $\vert\dot N(\S'_{\G},\C''_{\G})\cap\bar c''_{\O_i}\vert$ is homeomorphic to $M_{\O_i}$ and is a boundary component of $\vert\dot N(\S'_{\G},\C''_{\G})\cap\bar c''_{\L}\vert$, for $i=1,\ldots,m$.

Referring to the proof of Lemma~\ref{boundary0}, the simplex $\s^h=\langle \widehat L_0,\widehat L_1, \ldots\widehat L_h\rangle$ of $(V_{\L}\star\C'_{\L})'$ belongs to $\dot N(\S'_{\G},\C''_{\G})$ if and only if either (i) $\L$ is the last element of the chain $L_0$, and $\L\neq\L_{0,0}\neq\O_j$, for $j=1,\ldots,m$, or (ii) there exists $j$ such that $\O_j\neq\L_{0,0}$ is the last element of the chain $L_0$. Let $K$ be the complex containing the simplexes satisfying (i) and, for $i=1,\ldots,m$, let $K_i$ be the complex containing the simplexes satisfying (ii) and such that $\L_{0,0}$ is a residue of $\O_i$. Moreover, $\s^h$ belongs to $C(\S'_{\L},\C''_{\L})$ if and only if $\L$ is not an element of the chain $L_h$ and $\L_{0,0}\neq\O_j$, for $i=1,\ldots,m$. In particular, $\s^h$ belongs to $C(\S'_{\O_i},\C''_{\O_i})$ if and only if $\L$ is not an element of the chain $L_h$ and $\L_{0,0}$ is a residue of $\O_i$.

The map $\kappa:S^0(C(\S'_{\L},\C''_{\L}))\to S^0(K)$ defined by $\kappa(\widehat L)=\widehat{L'}$, where $L$ is the chain $\L_0\prec\L_1\prec\cdots\prec\L_k$, with $\L_k\prec\L$, and $L'$ is the chain $\L_0\prec\L_1\prec\cdots\prec\L_k\prec\L$, induces an isomorphism between $C(\S'_{\L},\C''_{\L})$ and $K$. Therefore, $\vert K\vert$ is homeomorphic to $M_{\L}=\vert C(\S'_{\L},\C''_{\L})\vert$. Moreover, for $i=1,\ldots,m$, the map $\kappa_i:S^0(C(\S'_{\O_i},\C''_{\O_i})\to S^0(K_i)$ defined in the same way of $\kappa$ with $L'$ being the chain $\L_0\prec\L_1\prec\cdots\prec\L_k\prec\L$ (resp. $\L_0\prec\L_1\prec\cdots\prec\L_k\prec\O_i$) if $\L_k=\O_i$ (resp. if $\L_k\prec\O_i$) induces an isomorphism between $C(\S'_{\O_i},\C''_{\O_i})$ and $K_i$. Therefore, $\vert K_i\vert$ is homeomorphic to $\vert C(\S'_{\O_i},\C''_{\O_i})\vert$, which is homeomorphic to $M_{\O_i}\times I$ by Lemma~1.22 of \cite{[Hu]}. Since $K\cap K_i$ is isomorphic to $\dot N(\S'_{\O_i},\C''_{\O_i})=\dot N(\S'_{\G},\C''_{\G})\cap\bar c''_{\O_i}$ via $\kappa_i$ and $\vert\dot N(\S'_{\O_i},\C''_{\O_i})\vert=M_{\O_i}$, then $\vert\dot N(\S'_{\G},\C''_{\G})\cap\bar c''_{\L}\vert$ is homeomorphic to $M_{\L}$. This concludes the proof.
\end{proof}

\medskip

When $\dim(\S_{\G})>1$ the description of the boundary of $M_{\G}$ is in general rather more involved. Note that Lemma~\ref{boundary0} (resp. Proposition~\ref{boundary1}) always applies when $\dim(M_{\G})=3$ (resp. when $\dim(M_{\G})=4$). 

Some properties of $M_{\G}$ and $\widehat M_{\G}$ correspond to properties of the representing graph $\G$. For example the following result generalizes Theorem~4 of \cite{[FGG]}:

\begin{prop}\label{bip}
The quasi-manifold $\widehat M_{\G}$, as well as the manifold $M_{\G}$, is orientable if and only if $\Gamma$ is bipartite.
\end{prop}

\begin{proof} Obviously any orientation on $\widehat M_{\G}$ restricts to an orientation on $ M_{\G}$. Furthermore, any orientation on $M_{\G}$ extends to an orientation on $\widehat M_{\G}$, since $|\bar \sigma''\cap C(\bar S''_{\G},\K''_{\G})|$ is an $n$-ball (resp. an $(n-1)$-ball), for any $n$-simplex (resp. $(n-1)$-simplex) $\sigma$ of $\K_{\G}$. Therefore it suffices to prove the statement for $\widehat M_{\G}$.
By construction any global orientation on $\widehat M_{\G}$ induces a local orientation for any maximal simplexes of $\K_{\G}$, which corresponds to one of the two classes of total orderings of the set $\D$, up to even permutations, in such a way that two $n$-simplexes having a common $(n-1)$-face are associated to different classes.
As a consequence, the graph $\G$ is bipartite since it can not have odd cycles. On the other hand, if $\G$ is bipartite with bipartition of vertices $V(\G)=V'\cup V''$, then orient any $n$-simplex associated to elements of $V'$ according with one class of orderings of $\D$ and any $n$-simplex associated to elements of $V''$ according with the other class. This choice provides a global orientation for $\widehat M_{\G}$. 
\end{proof}

\medskip

The computation of the Euler characteristic of $M_{\G}$, $\widehat M_{\G}$ and $|\S_{\G}|$ is a routine fact:

\begin{lemma}\label{chi} Let $\G$ be an $(n+1)$-colored graph. Then 
$$\chi(M_{\G})=\sum_{h=0}^n(-1)^h|\R'_h(\G)|,\  \chi(\widehat M_{\G})=\sum_{h=0}^n(-1)^{n-h}|\R_h(\G)|\ \textrm{ and }\  \chi(|\S_{\G}|)=\sum_{h=3}^n(-1)^{n-h}|\R''_h(\G)|.$$ 
%Moreover, the Euler characteristic of $\widehat M_{\G}$ is given by the recursive formula $$\chi(\widehat M_{\G})=\chi(M_{\G})+\vert\R''_h(\G)\vert-\sum_{\L\in\R''_h(\G)}\chi(\widehat M_{\L}).$$ 
%$\chi(\widehat M_{\G})=\chi(M_{\G})+\vert\partial M_{\G}\vert-\chi(\partial M_{\G})$, where $\vert\partial M_{\G}\vert$ denotes the number of connected components of $\partial M_{\G}$.
\end{lemma}

\begin{proof} The manifold $M_{\G}$ retracts to the space $X=\cup_{\L\in\R'(\G)}c_{\L}$ by Remark~\ref{cylinder}, and therefore $\chi(M_{\G})=\chi(X)$ can be obtained by the usual formula for the CW-complexes. On the other hand, both $\chi(\widehat M_{\G})$ and $\chi(|\S_{\G}|)$ can be obtained from the pseudo-simplicial complex $\K_{\G}$ by considering that any $h$-residue of the graph corresponds to an $(n-h)$-simplex of $\K_{\G}$.
\end{proof}

\section{\hskip -0.7cm . Graph suspension}\label{section_suspension}

If $\G$ is an $n$-colored graph, for a fixed $c\in\Delta_{n-1}$ define $\Sigma_c(\G)$ as the $(n+1)$-colored graph obtained from $\G$ by adding a set of $n$-edges parallel to the $c$-edges of $\G$. We refer to $\Sigma_c(\G)$ as the {\it $c$-suspension} of $\G$. The following result states that this construction is strictly related to the suspension construction of PL spaces (in the following $\Sigma(\cdot)$ denotes the suspension of either a PL space or a simplicial complex).

\begin{thm} \label{suspension} Let $\G$ be an $n$-colored graph and $c$ be any color in $\Delta_{n-1}$, then:
\begin{itemize}
\item[•] (i) $\widehat M_{\Sigma_c(\G)}=\Sigma(\widehat M_{\G})$; 
\item[•] (ii) $\vert\S_{\Sigma_c(\G)}\vert=\vert\Sigma(\S_{\G})\vert$, if $\widehat M_{\G}$ is not a sphere;
\item[•] (iii) $M_{\Sigma_c(\G)}=M_{\G}\times I$, if $\widehat M_{\G}$ is not a sphere.
\end{itemize}
\end{thm}

\begin{proof} Let $\O=\Sigma_c(\G)$ and denote by $\R^c(\G)$ the set of all residues of $\G$ containing $c$-edges. Each $\Delta$-residue $\L$ of $\G$ can be considered as a $\Delta$-residue of $\O$, as well as $\G$. So $\O$ has a unique $\widehat {n}$-residue $\G$, and a unique $\widehat c$-residue $\G''$, which is isomorphic to $\G$. Moreover, if $\L\in\R^c(\G)$ let us call $\L'$ (resp. $\L''$) the $(\Delta\cup\{n\})$-residue of $\O$ (resp. the $(\Delta\cup\{n\}-\{c\})$-residue of $\O$) having the same vertices of $\L$. 

(i) Consider the $n$-dimensional simplicial complex $K=\Sigma(\C'_{\G})=\{P',P''\}\star \C'_{\G}$, then $S^0(K)=\{P',P''\}\cup_{h=0}^{n-1}\{V_{\L}\mid \L\in\R_h(\G)\}$. We will show by induction that a suitable subdivision of $K$ is isomorphic to $\C'_{\O}$. 

The induction is performed on the $h$-residues of $\R^c(\G)$, for $h=1,\ldots,n-1$.
The first step of the induction is to obtain a stellar subdivision $K_1$ of $K$, by 
starring the 1-simplex $\langle P',V_e\rangle$ at an internal point $V'_e$ and the 1-simplex $\langle P'',V_e\rangle$ at an internal point $V''_e$, for any $c$-edge $e$ of $\G$. It is obvious that $K_1$ does not depend on the order of the above starrings.  After this step the 1-simplex $\langle P',V_e\rangle$ (resp. $\langle P'',V_e\rangle$) is subdivided into the two 1-simplexes $\langle P',V'_e\rangle$ and $\langle V_e, V'_e\rangle$ (resp. $\langle P'',V''_e\rangle$ and $\langle V_e, V''_e\rangle$) and the maximal simplex $\langle V_v,V_e,V_{\L_2},\ldots,V_{\L_{n-1}},P'\rangle$ (resp. $\langle V_v,V_e,V_{\L_2},\ldots,V_{\L_{n-1}},P''\rangle$) of $K$ is subdivided into the two maximal simplexes $\langle V_v,V'_e,V_{\L_2},\ldots,V_{\L_{n-1}},P'\rangle$ and $\langle V_v,V'_e,V_e,V_{\L_2},\ldots,V_{\L_{n-1}}\rangle$ (resp. $\langle V_v,V''_e,V_{\L_2},\ldots,V_{\L_{n-1}},P''\rangle$ and $\langle V_v,V''_e,V_e,V_{\L_2},\ldots,V_{\L_{n-1}}\rangle$).

Suppose the first $h-1$ induction steps performed, producing the subdivision $K_{h-1}$ of $K$. We define the induction step $h$ as follows: produce the complex $K_h$ as the stellar subdivision of $K_{h-1}$ obtained by starring  both the 1-simplex $\langle P',V_{\L_h}\rangle$ at an internal point $V'_{\L_h}$ and the 1-simplex $\langle P'',V_{\L_h}\rangle$ at an internal point $V''_{\L_h}$, for any $h$-residue $\L_h$ of $\R^c(\G)$. Again the result of these starrings do not depend on their order. After this step the 1-simplex $\langle P',V_{\L_h}\rangle$ (resp. $\langle P'',V_{\L_h}\rangle$) is subdivided into the two 1-simplexes $\langle P',V'_{\L_h}\rangle$ and $\langle V_{\L_h}, V'_{\L_h}\rangle$, (resp. $\langle P'',V''_{\L_h}\rangle$ and $\langle V_{\L_h}, V''_{\L_h}\rangle$). 
The maximal simplex $\langle V_{\L_0},\ldots, V_{\L_j},V'_{\L_{j+1}},\ldots, V'_{\L_{h-1}},V_{\L_h},V_{\L_{h+1}},\ldots,V_{\L_{n-1}},P'\rangle$ 
of $K_{h-1}$ is subdivided into the two maximal simplexes $\langle V_{\L_0},\ldots, V_{\L_j},V'_{\L_{j+1}},\ldots, V'_{\L_{h-1}},V'_{\L_h},V_{\L_{h+1}},\ldots,V_{\L_{n-1}},P'\rangle$ and $\langle V_{\L_0},\ldots, V_{\L_j},V'_{\L_{j+1}},\ldots, V'_{\L_{h-1}},V'_{\L_h},V_{\L_h},V_{\L_{h+1}},\ldots,V_{\L_{n-1}}\rangle$, where $0\le j\le h-1$. 
Analogously, the maximal simplex $\langle V_{\L_0},\ldots, V_{\L_j},V''_{\L_{j+1}},\ldots, V''_{\L_{h-1}},V_{\L_h},V_{\L_{h+1}},\ldots,V_{\L_{n-1}},P''\rangle$ is subdivided into the two maximal simplexes $\langle V_{\L_0},\ldots, V_{\L_j},V''_{\L_{j+1}},\ldots, V''_{\L_{h-1}},V''_{\L_h},V_{\L_{h+1}},\ldots,V_{\L_{n-1}},P''\rangle$ and  $\langle V_{\L_0},\ldots, V_{\L_j},V''_{\L_{j+1}},\ldots, V''_{\L_{h-1}},V''_{\L_h},V_{\L_h},V_{\L_{h+1}},\ldots,V_{\L_{n-1}}\rangle$.

At the end of the inductive process we obtain a complex $K_{n-1}$, which is a subdivision of $K$, having vertex set $S^0(K_{n-1})=S^0(K)\cup\{V'_{\L}\mid \L\in\R^c(\G)\}\cup\{V''_{\L}\mid \L\in\R^c(\G)\}$.

Let $\phi: S^0(K_{n-1})\to S^0(\O)$ be the map defined by $\phi(P')=V_{\G}$, $\phi(P'')=V_{\G''}$, $\phi(V_{\L})=V_{\L'}$ if $\L\in\R^c(\G)$ and 
$\phi(V_{\L})=V_{\L}$ if $\L\not\in\R^c(\G)$, $\phi(V'_{\L})=V_{\L}$ and $\phi(V''_{\L})=V_{\L''}$, for any $\L\in\R^c(\G)$. Then $\phi$ is clearly a bijection and induces an isomorphism between $K_{n-1}$ and $\C'_{\O}$.
 
In order to prove that, let $\sigma=\langle V_{\L_0},V_{\L_1},\ldots,V_{\L_{n-1}},V_{\L_n}\rangle$ be a maximal simplex of $\C'_{\O}$, where $\L_i$ is a $D_i$-residue of $\O$ such that $|D_i|=i$, for $i=0,\ldots,n$. 
If $D_n=\widehat n$ and $h=\min\{i\mid c\in D_i\}$ (resp. $D_n=\widehat c$ and $h=\min\{i\mid n\in D_i\}$), then $\sigma$ is the image via $\phi$ of the maximal simplex $\langle V_{\L_0},\ldots,V_{\L_{h-1}}, V'_{\L_h},V'_{\L_{h+1}},\ldots,V'_{\L_{n-1}},P'\rangle$ (resp. $\langle V_{\L_0},\ldots,V_{\L_{h-1}}, V''_{\tilde\L_h},V''_{\tilde\L_{h+1}},\ldots,V''_{\tilde\L_{n-1}},P''\rangle$, where $\tilde\L_{i}$ is the residue of $\G$ such that $\tilde\L_{i}''=\L_i$, for $i=h,\ldots,n-1$). 
Otherwise, let $D_n=\widehat j$ with $j\neq c,n$. Define $h'=\min\{i\mid c\in D_i\}$ and $h''=\min\{i\mid n\in D_i\}$.
If $h'<h''$ then $\sigma$ is the image via $\phi$ of the maximal simplex $\langle V_{\L_0},\ldots,V_{\L_{h'-1}}, V'_{\L_{h'}},V'_{\L_{h'+1}},\ldots,V'_{\L_{h''-1}}, V_{\tilde\L_{h''-1}},V_{\tilde\L_{h''}},\ldots,V_{\tilde\L_{n-1}}\rangle$, where $\tilde\L_{i}$ is the residue of $\G$ such that $\tilde\L_i'=\L_{i+1}$ for $i=h''-1,\ldots,n-1$. If $h''<h'$ then $\sigma$ is the image via $\phi$ of the maximal simplex $\langle V_{\L_0},\ldots,V_{\L_{h''-1}}, V''_{\L_{h''}},V''_{\L_{h''+1}},\ldots,V''_{\L_{h'-1}}, V_{\tilde\L_{h'-1}},V_{\tilde\L_{h'}},\ldots,V_{\tilde\L_{n-1}}\rangle$, where $\tilde\L_{i}$ is the residue of $\G$ such that $\tilde\L_{i}''=\L_i$ for $i=h'',\ldots,h'-1$, and $\tilde\L_{i}'=\L_{i+1}$ for $i=h'-1,\ldots,n-1$.

(ii) If $\sigma\in\C'_{\G}$ (possibly $\sigma=\emptyset$) then $\textrm{Link}(P'\star\sigma,K)=\textrm{Link}(P''\star\sigma,K)=\textrm{Link}(\sigma,\C'_{\G})$ (recall that $\textrm{Link}(\emptyset,\C'_{\G})=\C'_{\G}$). 
Therefore $\sigma$ is singular in $\C'_{\G}$ if and only if both $P'\star\sigma$ and $P''\star\sigma$ are singular in $K$. This proves that $S_K=\{P',P''\}\star \S_{\Gamma}=\Sigma(\S_{\Gamma})$ if  $\vert\C'_{\G}\vert=\widehat M_{\G}$ is not a sphere and $S_K=\emptyset$ if $\widehat M_{\G}$ is a sphere. Now the statement follows from the fact that the singular set of a complex is invariant under subdivisions.

(iii) By the previous points we can consider $M_{\Sigma_c(\G)}$ as $\vert K\vert-\textrm{int}(N(S_{K}))$. The complexes $H_1=P'\star \C'_{\G}$ and $H_2=P''\star \C'_{\G}$ are isomorphic, as well as $S_{H_1}$ and $S_{H_2}$. Therefore, it suffices to prove that $M_1=\vert H_1\vert-\textrm{int}(N(S_{H_1})) =M_{\G}\times I$.

By (ii), if $\widehat M_{\G}$ is not a sphere then $S_{H_1}= P'\star\S_{\G}$, which is a full subcomplex of $H_1$ containing $P'$.
As a consequence, $N(S_{H_1})=\vert N(S'_{H_1},H_1')\vert$ and $M_1=\vert C(S'_{H_1},H_1')\vert=\vert\cup_{v\in(S^0(H_1)-S^0(S_{H_1}))}\textrm{Star}(v,H_1')\vert=\vert\cup_{v\in(S^0(\C'_{\G})-S^0(\S_{\G}))}\textrm{Star}(v,H_1')\vert$. The complex $D=C(S'_{H_1},H_1')\cap\textrm{Star}(P',H_1')$ is isomorphic to $C(\S_{\G},\C''_{\G})$, $L=(P'\star D)\cup C(S'_{H_1},H_1')$ is a subdivision of $P'\star C(\S'_{\G},\C''_{\G})$ and $\textrm{Star}(P',L)\cap C(\S'_{\G},\C''_{\G})=\emptyset$. Then Lemma~1.22 of \cite{[Hu]} applies and $M_1=\vert C(S'_{H_1},H_1')\vert=\vert P'\star C(\S'_{\G},\C''_{\G})\vert-\textrm{int}\vert\textrm{Star}(v,L)\vert$ is homeomorphic to $\vert C(\S'_{\G},\C''_{\G})\vert\times I=M_{\G}\times I$.
\end{proof}

\medskip

It is noteworthy that Theorem~\ref{suspension} extends to the whole class of colored graphs a result obtained in \cite{[FG2]} for the case in which $\widehat M_{\G}$ is a closed manifold. Actually, a graph representing $\Sigma(\widehat M_{\G})$ is easily obtainable from $\G$ by doubling it, namely taking two isomorphic copies of $\G$ and joining the corresponding vertices with an $n$-edge. The importance of the previous result relies in the fact that the graphs $\G$ and $\Sigma_c(\G)$ have the same order.

As a relevant example of the suspension graph construction, Figure~1 shows a 6-colored graph $\G$ which is the double suspension of a 4-colored graph representing the Poincar\`e homology sphere depicted in \cite{[LM]}. By a remarkable result by J. W. Cannon (see \cite{[Ca]}), $\widehat M_{\G}$ is topologically homeomorphic to $S^5$ despite that, with the structure induced by $\C'_{\G}$, it is not PL-homeomorphic to $S^5$. In fact, $|\S_{\G}|$ is not empty and it is homeomorphic to $S^1$.

\begin{figure}[h!]                   
\begin{center}                         
\includegraphics[width=15cm]{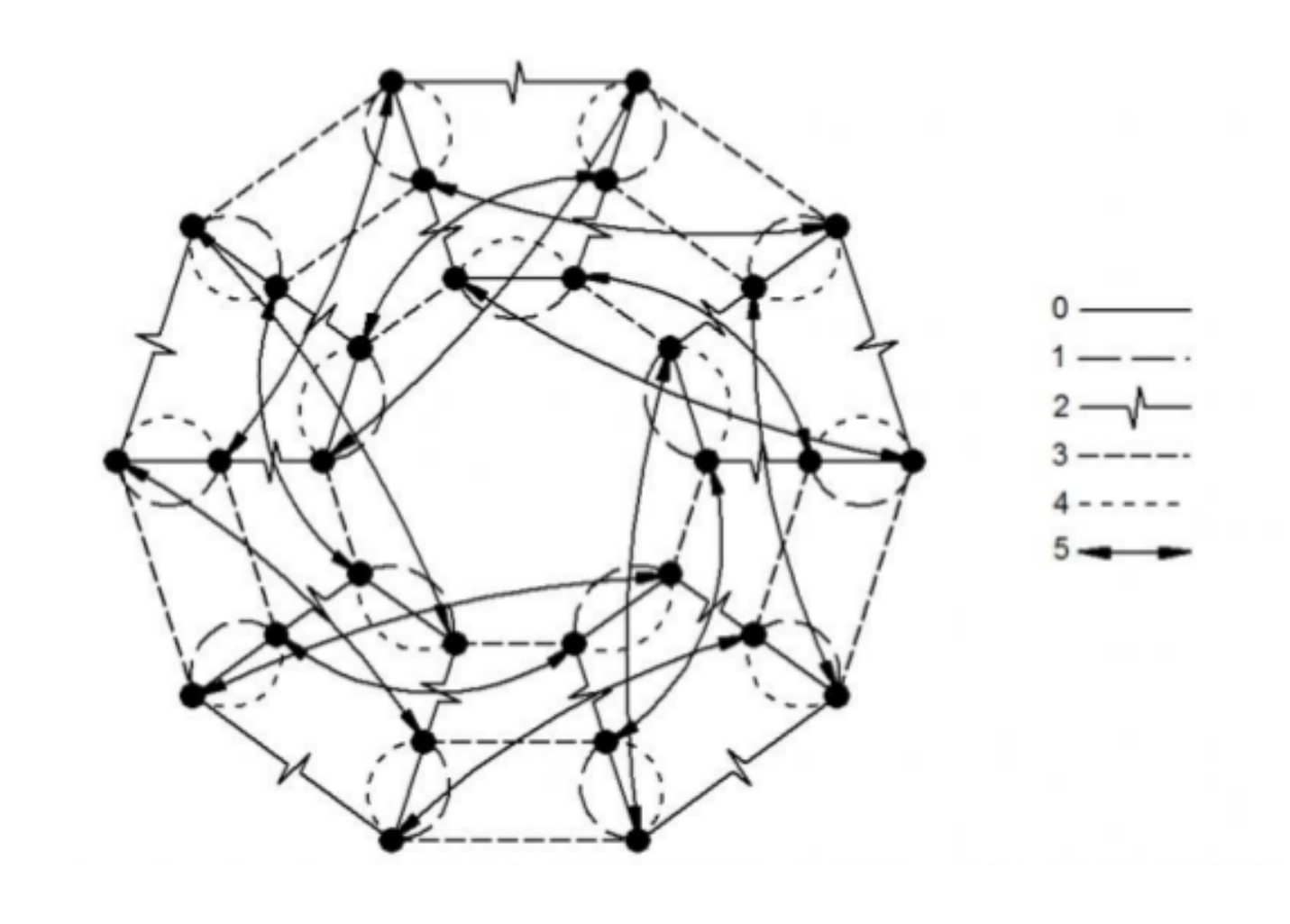}
\caption[legenda elenco figure]{A graph representing the double suspension of the Poincar\`e homology sphere.}\label{fig0}
\end{center}
\end{figure}

\begin{rem}
From the proof of Proposition~\ref{existence} it follows that any compact connected $n$-manifold admits a representation by an $(n+1)$-colored graph where all the singular residues (if any) are $n$-residues. Nevertheless, the representation with general $(n+1)$-colored graphs may be more economical in terms of the order of the representing graph. For example, the 4-manifold $M=S^1\times S^1\times B^2$ can be represented by the order six 5-colored graph $\G'=\Sigma_2(\Sigma_1(\G))$ depicted in Figure~2, where $\G$ is the standard order six graph representing $S^1\times S^1$. But it is easy to realize that $M$ does not admit any representation by 5-colored graphs of order less than 24 and without singular 3-residues, since in this case it would admit a 4-residue representing $\partial (S^1\times S^1\times B^2)=S^1\times S^1\times S^1$, and it is well known that $S^1\times S^1\times S^1$ does not admits representation by 4-colored graphs with less than 24 vertices (see \cite{[Li]}).
\end{rem}

\begin{figure}[h!]                   
\begin{center}                         
\includegraphics[width=10cm]{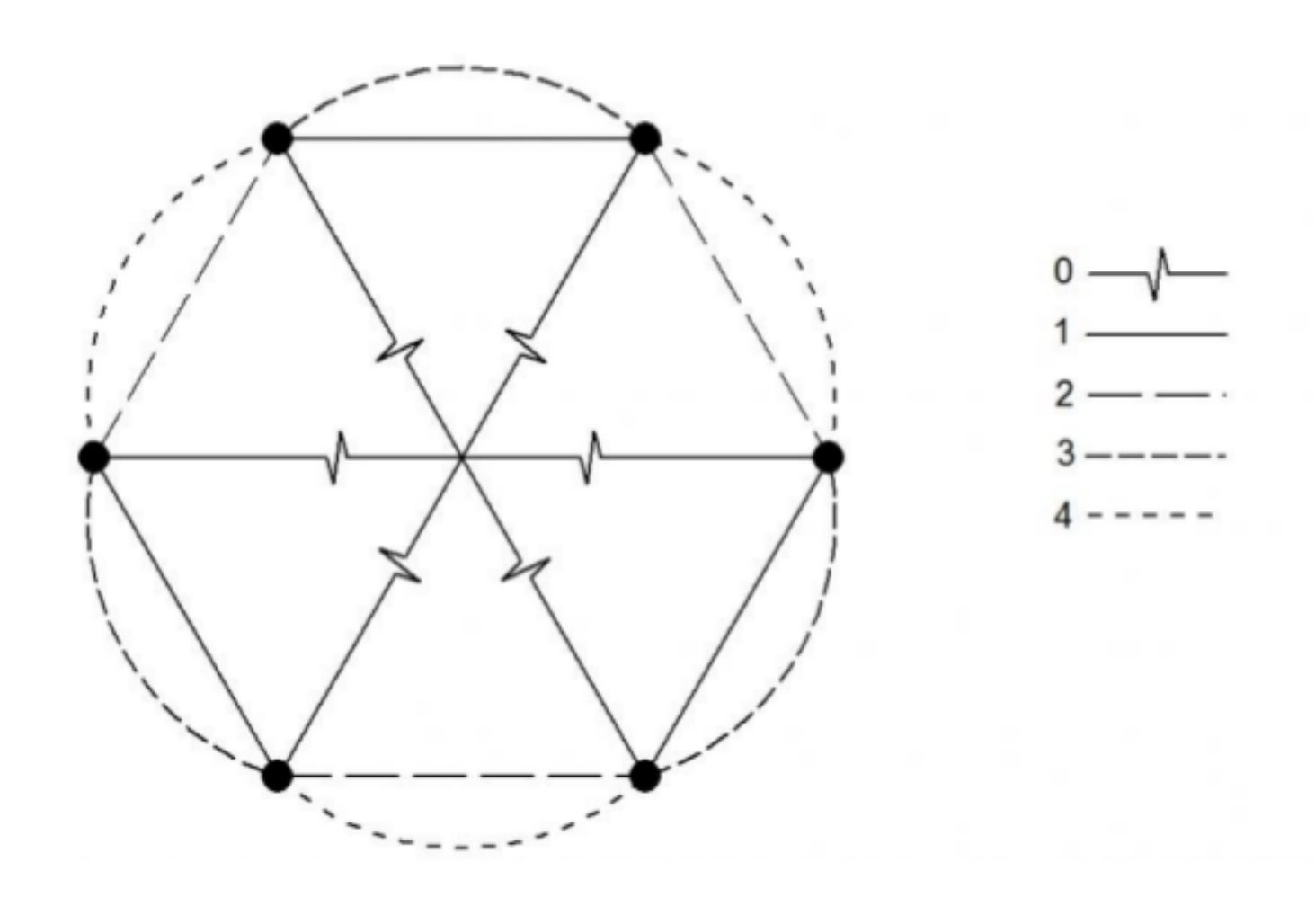}
\caption[legenda elenco figure]{A graph representing $S^1\times S^1\times B^2$.}\label{fig1}
\end{center}
\end{figure}

%\medskip

It is easy to see that an $(n+1)$-colored graph of order two always represents an $n$-sphere, since it is the $(n-2)$-th suspension of the order two bigon, which represents $S^1$. Also for graphs of order 4 the characterization is easy.

\begin{prop} \label{order_4} Let $\G$ be an $(n+1)$-colored graph of order 4, then:
\begin{itemize}
\item[•] (i) $\widehat M_{\G}=M_{\G}=S^n$ if $\G$ is bipartite;
\item[•] (ii) $\widehat M_{\G}=\Sigma^{n-2}(RP^2)$ and $M_{\G}=RP^2\times B^{n-2}$  if $\G$ is non-bipartite. 
\end{itemize}
\end{prop}

\begin{proof} Since $\G$ is connected it contains at least a bigon of order 4, and we can suppose, up to isomorphism, that it is a $\{0,1\}$-residue.
If $n=1$ the graph $\G$ is a bigon and $\widehat M_{\G}=M_{\G}=S^1$. If $n>1$ and $\G$ is bipartite, then any pair of $c$-edges, for $c=2,\ldots,n$ is parallel to either the pair of 0-edges or the pair of 1-edges. By Theorem~\ref{suspension} we obtain (i). On the contrary, if $\G$ is not bipartite then it contains a cycle of order three. As a consequence, there is at least a $3$-residue $\L$ of $\G$ (say a $\{0,1,2\}$-residue) which is isomorphic to the complete graph of order 4, and therefore $\widehat M_{\L}=M_{\L}=RP^2$. Then any pair of $c$-edges, for $c=3,\ldots,n$ is parallel to either the pair of 0-edges or the pair of 1-edges or the pair of 2-edges, and the statement follows from Theorem~\ref{suspension}.
\end{proof}

\section{\hskip -0.7cm . Dipole moves}  \label{moves}

Given an $(n+1)$-colored graph $\G$, an \textit{$h$-dipole} ($1\leq h\leq n$) \textit{involving colors} $c_1,\ldots,c_h\in\D$ of $\G$ is a subgraph $\theta$ of $\G$ consisting of two vertices $v'$ and $v''$ joined by exactly $h$ edges, colored by $c_1,\ldots,c_h$, such that $v'$ and $v''$ belong to different $\widehat{\{ c_1,\ldots, c_h\}}$-residues of $\G$.

By \textit{cancelling} $\theta$ \textit{from} $\G$, we mean to remove $\theta$ and to paste together the hanging edges according to their colors, thus obtaining a new $(n+1)$-colored graph $\G'$. Conversely, $\G$ is said to be obtained from $\G'$ by \textit{adding} $\theta$. A \textit{dipole move} is either the cancellation or the addition of a dipole.

An $h$-dipole $\theta$ of $\G$ is called \textit{proper} when $\widehat M_{\G}$ is homeomorphic to $\widehat M_{\G'}$ (see \cite{[G1]}). As a consequence, if two colored graphs $\G_1$ and $\G_2$ are connected by a sequence of proper dipole moves then $\widehat M_{\G_1}=\widehat M_{\G_2}$, as well as $M_{\G_1}=M_{\G_2}$. Furthermore, $\theta$ is called {\it singular} if the two $\widehat{\{ c_1,\ldots, c_h\}}$-residues containing the vertices of $\theta$ are both singular. Otherwise the dipole is called {\it ordinary}.

\begin{prop}\label{proper} \cite{[G1]} Any ordinary dipole of an $(n+1)$-colored graph $\G$ is proper. As a consequence, all $n$- and $(n-1)$-dipoles of $\G$ are proper. 
\end{prop}

If $\G$ represents a closed $n$-manifold all dipoles of $\G$ are proper and Casali proved in \cite{[C$_1$]} that dipole moves are sufficient to connect different $4$-colored graphs representing the same closed $3$-manifold. This result is no longer true in the case with boundary, even in dimension three (see \cite{[CM]}).

It seems natural to argue that the converse of Proposition~\ref{proper} also holds, but we are able to prove it only for singular manifolds (see Corollary~\ref{singular}). 

\begin{lemma}\label{moves singular} If $\G'$ is obtained from $\G$ by cancelling a dipole $\theta$, then $|\S_{\G'}|$ is homeomorphic to a quotient of $|\S_{\G}|$.
\end{lemma}

\begin{proof} Topologically, the effect of the cancellation of the $h$-dipole $\theta$ involving colors $c_1,\ldots,c_h$ is the following: $\widehat M_{\G'}$ is obtained from $\widehat M_{\G}$ by removing the open $n$-ball $\textrm{int}(B)$ where $B=\vert\textrm{Star}(v',\C'_{\G})\cup\textrm{Star}(v'',\C'_{\G})\vert$ and afterthat attaching the $(n-1)$-ball $\partial B\cap\partial\vert \textrm{Star}(v',\C'_{\G})\vert$ with the $(n-1)$-ball $\partial B\cap\partial\vert \textrm{Star}(v'',\C'_{\G})\vert$ via the identification of any simplex $\s'=\langle V_{\L'_0}, V_{\L'_1},\ldots,V_{\L'_k}\rangle$ with the simplex $\s''=\langle V_{\L''_0}, V_{\L''_1},\ldots,V_{\L''_k}\rangle$, where $\L'_i$ and $\L''_i$ are $\Delta_{(i)}$-residues containing $v'$ and $v''$ respectively, with $ \emptyset\neq\Delta_{(0)}\not\subseteq\{ c_1,\ldots, c_h\}$, for $i=0,1,\ldots,k$. 
The only vertices of $\C'_{\G}$ belonging to $\textrm{int}(B)$ are the cone vertices associated either to $\theta$ or to a residue of $\theta$. Since $\widehat M_{\theta}$ is a sphere, these vertices do not belong to $\S_{\G}$. As  a consequence,  $\textrm{int}(B)\cap\vert\S_{\G}\vert=\emptyset$ and $|\S_{\G'}|$ is obtained from $|\S_{\G}|$ just by considering the attachings of $\s'$ with $\s''$ when they are both simplexes of $\S_{\G}$.
\end{proof}

\medskip

\begin{lemma}\label{lemma} Let $\theta$ be a singular $h$-dipole involving colors $c_1,\ldots,c_h$ of an $(n+1)$-colored graph $\G$. If for any $\Delta\subset \widehat{\{ c_1,\ldots, c_h\}}$ at least one of the two $\Delta$-residues containing the vertices of $\theta$ is ordinary, then $\theta$ is not proper.
\end{lemma}

\begin{proof} Let $v',v''$ be the endpoints of $\theta$ and let $\G'$ be the $(n+1)$-colored graph obtained from $\G$ by cancelling $\theta$. 
Then $\S_{\G'}$ is obtained from $\S_{\G}$ by attaching each simplex $\langle V_{\L'_0}, V_{\L'_1},\ldots,V_{\L'_k}\rangle$ with the simplex $\langle V_{\L''_0}, V_{\L''_1},\ldots,V_{\L''_k}\rangle$, where $\L'_0$ and $\L''_0$ are the $\widehat{\{ c_1,\ldots, c_h\}}$-residues containing $v'$ and $v''$ respectively. Topologically, the operation consists in the attaching of two $(h-1)$-balls with common boundary. So $\chi(|\S_{\G'}|)=\chi(|\S_{\G}|)+(-1)^{h}$ and therefore $\vert\S_{\G'}\vert$ is not homeomorphic to $\vert\S_{\G}\vert$. As a consequence, $\widehat M_{\G'}$ is not homeomorphic to $\widehat M_{\G}$ and the dipole is not proper.
\end{proof}

%\medskip

\begin{cor}\label{singular} Let $\G$ be an $(n+1)$-colored graph such that $\widehat M_{\G}$ is a singular manifold, then a dipole of $\G$ is proper if and only if it is ordinary.
\end{cor}

\begin{proof} By Lemma~\ref{singular manifold} any singular residue of $\G$ is an $n$-residue and contains no singular residues, therefore Lemma~\ref{lemma} holds.
\end{proof}

\medskip

For the manifold with boundary $M_{\G}$ we have the following consequence of the previous results.

\begin{prop}\label{partial} Let $\G$ be an $(n+1)$-colored graph such that $\widehat M_{\G}$ is a singular manifold, and let $\G'$ be the graph obtained from $\G$ by cancelling a dipole $\theta$. Then $M_{\G'}=M_{\G}$ if and only if $\theta$ is ordinary.
\end{prop}

\begin{proof} If $\theta$ is ordinary then $\widehat M_{\G'}=\widehat M_{\G}$ by Proposition~\ref{proper} and therefore $M_{\G'}=M_{\G}$. If $\theta$ is singular then it is a 1-dipole and consequently $|\R''_n(\G')|<|\R''_n(\G)|$. Since in this case $|\R''_n(\G)|$ (resp. $|\R''_n(\G')|$) is exactly the number of boundary components of $M_{\G}$ (resp. of $M_{\G'}$), then $M_{\G'}\neq M_{\G}$. 
\end{proof}

\medskip

A vertex $v\in V(\G)$ is called an {\it internal vertex} if all $n$-residues containing $v$ are ordinary, otherwise it is called a {\it boundary vertex}. The {\it index} of $v$ is the number of singular $n$-residues of $\G$ containing $v$. So an internal vertex has index 0 and a boundary vertex has index $r$, with $1\le r\le n+1$. 

Some useful properties follow from the previous results.

\begin{lemma} \label{no-dipole} Let $M$ be a compact connected $n$-manifold without spherical boundary components, then:
\begin{itemize}
\item[•] (i) $M$ can be represented by an $(n+1)$-colored graph with no ordinary dipoles;
\item[•] (ii)  $M$ can be represented by an $(n+1)$-colored graph with at least one internal vertex.
\item[•] (iii) if $\partial M\neq\emptyset$ then $M$ can be represented by an $(n+1)$-colored graph with at least one boundary vertex of index one.
\end{itemize}
\end{lemma}

\begin{proof} Let $\G$ be an $(n+1)$-colored graph representing $M$.

(i) If $\G$ has an ordinary dipole $\theta$, then the dipole is proper and by cancelling it we obtain a new $(n+1)$-colored graph still representing $M$. A finite sequence of such cancellations of ordinary dipoles obviously yields an $(n+1)$-colored graph representing $M$ and without ordinary dipoles.

(ii, iii) If $\partial M=\emptyset$ there is nothing to prove. Otherwise, let $v$ be a boundary vertex with minimal index $r>0$ and let $c\in\D$ be such that the $\widehat c$-residue containing $v$ is singular. By adding an $n$-dipole along the $c$-edge
containing $v$ we obtain two new vertices, $v'$ and $v''$, which
are both singular of order $r-1$. In fact, the $\widehat c$-residue
containing them is obviously ordinary (it is the standard $n$-colored graph representing $S^{n-1}$), and for each $d\in\widehat c$ any $\widehat d$-residue containing them is singular if and only if the $\widehat d$-residue
containing $v$ in $\G$ is singular. So by induction on $r$ we can obtain an internal vertex (resp. a boundary vertex of index one) in not more than $r$ steps (resp. $r-1$ steps).
\end{proof}

\medskip

%As a consequence, any closed $n$-manifold admits representation by supercontracted graphs.

It is important to note that, for getting a minimal representation of a manifold in terms of the order of the representing graph, we can consider only colored graphs without ordinary dipoles.

\section{\hskip -0.7cm . Fundamental group} 

If $\G$ is an $(n+1)$-colored graph, with $n>1$, then the fundamental group of the
manifold $M_{\G}$ coincides with the fundamental group of the
associated 2-dimensional polyhedron $\G^{(2)}$, since $M_{\G}$ is obtained by attaching to $\G^{(2)}$ pieces which are retractable (in virtue of Remark~\ref{cylinder}) and $h$-balls, for $3\le h\le n$.
Therefore, the computation of $\pi_1(M_{\G})$ is a routine fact: a finite presentation for it has generators corresponding to the edges which are not in a fixed spanning tree of $\G$ and relators corresponding to all 2-residues of $\G$. The fundamental group of the quasi-manifold $\widehat M_{\G}$ is a quotient of the one of $M_{\G}$, since retractable pieces are replaced by cones which kill some elements of $\pi_1(M_{\G})$.

In several cases the two groups can be obtained by selecting a particular class of edges and 2-residues, as follows.
If $c\in\D$, define the \textit{$c$-group} of $\G$ as
the group $\pi(\G,c)$ generated by all $c$-edges (with a fixed arbitrary
orientation) and with relators corresponding to all $\{i,c\}$-residues, for any
$i\in \widehat c$, obtained in the following way: give an orientation to each involved 2-residue, choose a starting vertex and follow the bigon according
to the chosen orientation. The relator is obtained by taking the $c$-edges of
the bigon in the order they are reached in the path, with the
exponent $+1$ or $-1$ according to whether the orientation of the
edge is coherent or not with the one of the bigon.

In general $\pi(\G,c)$ depends on $c$, but when $c$ is an ordinary color\footnote{A color $c\in\D$ is called {\it ordinary} if $\G$ has no singular $\widehat c$-residues, otherwise $c$ is called {\it singular}.} the group is strictly connected with the fundamental group
of $M_{\G}$ (see \cite{[Li]} for closed 3-manifolds and \cite{[Gr]} for closed $n$-manifolds).

\begin{prop} \label{groupM} Let $\G$ be an $(n+1)$-colored graph and $c$ be an ordinary color for $\G$. Then $\pi_1(M_{\G})$ is the quotient of $\pi(\G,c)$,
obtained by adding to the relators a minimal set of $c$-edges which connect 
the graph $\G_{\widehat c}$.
\end{prop}

\begin{proof}
The group $\pi_1(M_{\G})=\pi_1(\G^{(2)})$ is isomorphic to the fundamental group of the space $X$ obtained by adding to $\G^{(2)}$ the $n$-balls corresponding to the $\widehat c$-residues. The space $X$ has the same homotopy type of a 2-complex with 0-cells corresponding to the $\widehat c$-residues, 1-cells corresponding to the $c$-edges of $\G$ and 2-cells corresponding to the $\{c,i\}$-residues of $\G$, for $i\in\widehat c$. So the result is straightforward.
\end{proof}

\begin{cor} \label{groupM-cor} Let $\G$ be an $(n+1)$-colored graph and $c$ be an ordinary
color for $\G$ such that $g_{\widehat c}=1$, then $\pi_1(M_{\G})\cong \pi(\G,c)$.
\end{cor}

When $\G$ has no more than a singular color, we have the following characterization of the fundamental group of $\widehat M_{\G}$.\footnote{The result was first proved in \cite{[Ch]} by using the dual construction.}

\begin{prop} \label{grouphatM} Let $\G$ be an $(n+1)$-colored graph and let $c\in\D$ be such that any color different from $c$ is ordinary. Then $\pi_1(\widehat M_{\G})$ is the quotient of $\pi(\G,c)$, obtained by adding to the relators a minimal set of $c$-edges which connect the graph $\G_{\widehat c}$.
\end{prop}

\begin{proof} 
The group $\pi_1(\widehat M_{\G})$ is isomorphic to the fundamental group of the space $X$ obtained from $\G^{(2)}$ by performing cone constructions corresponding to the $\widehat c$-residues and all their residues. 
%In fact, $\widehat M_{\G}$ can be obtained from $X$ by adding $h$-balls, with $3\le h\le n$. 
Since $c_{\L}$ is a cone over $\L^{(n-1)}$, for any $\widehat c$-residue $\L$, the space $X$ has the same homotopy type of a 2-complex with 0-cells corresponding to the $\widehat c$-residues, 1-cells corresponding to the $c$-edges of $\G$ and 2-cells corresponding to the $\{c,i\}$-residues of $\G$, for any $i\in\widehat c$. So the result is straightforward. 
\end{proof}

\begin{cor} \label{grouphatM-cor} Let $\G$ be an $(n+1)$-colored graph and $c\in\D$. If $g_{\widehat c}=1$ and any color different from $c$ is ordinary,
then $\pi_1(\widehat M_{\G})\cong \pi(\G,c)$.
\end{cor}

\section{\hskip -0.7cm . Connected sums} \label{sum}

Suppose that $\G'$ and $\G''$ are two $(n+1)$-colored graphs and let $v'\in V(\G')$ and $v''\in V(\G'')$. We can construct a new $(n+1)$-colored graph $\G$,
called \textit{connected sum} of $\G'$ and $\G''$ (along $v'$ and $v''$), and
denoted by $\G=\G' _{v'}\#_{v''}\G''$, by removing the vertices
$v'$ and $v''$ and by welding the resulting hanging edges with the same color.

In general, the connected sum of two $(n+1)$-colored graphs depends on the choice of the cancelled vertices. But when these vertices are either internal or boundary vertices of index one with respect to $n$-residues of the same colors (the latter condition always holds, up to a color permutation in one of the two graphs), then the connected sum of the graphs is strictly connected with the connected sum of the represented manifolds.

\begin{prop} Let $\G',\G''$ be $(n+1)$-colored graphs and $v'\in V(\G'),v''\in V(\G'')$.
\begin{itemize}
\item[•] (i) if $v'$ and $v''$ are both internal vertices, then $M_{\G' _{v'}\#_{v''} \G''}=M_{\G'}\#M_{\G''};$
\item[•] (ii) if $v'$ and $v''$ are both boundary vertices of index one, each belonging to a singular $\widehat c$-residue, then  $M_{\G' _{v'}\#_{v''} \G''}=M_{\G'}\#_{\partial}M_{\G''}$, where the boundary connected sum of the manifolds is performed along the boundary components corresponding to the involved $\widehat c$-residues.
\end{itemize}
\end{prop}

\begin{proof} Let $\G=\G' _{v'}\#_{v''}\G''$. Then the complex $\C'_{\G}$ is obtained from $(\C'_{\G'}-\textrm{Star}(v',\C'_{\G'}))\cup(\C'_{\G''}-\textrm{Star}(v'',\C'_{\G''}))\cup\textrm{Link}(v',\C'_{\G'})\cup\textrm{Link}(v'',\C'_{\G''})$ by attaching any simplex $\langle V_{\L'_0},V_{\L'_1},\ldots,V_{\L'_h}\rangle$ of $\textrm{Link}(v',\C'_{\G'})$ with the simplex $\langle V_{\L''_0},V_{\L''_1},\ldots,V_{\L''_h}\rangle$ of $\textrm{Link}(v'',\C'_{\G''})$, where $\L'_i$ and $\L''_i$ are both $D_i$-residues, the first one of $\G'$ containing $v'$ and the second one of $\G''$ containing $v''$. Therefore, $\widehat M_{\G}$ is obtained by removing from $\widehat M_{\G'}$ and $\widehat M_{\G''}$ the $n$-balls $\textrm{int}(\vert\textrm{Star}(v',\C'_{\G'})\vert)$ and $\textrm{int}(\vert\textrm{Star}(v'',\C'_{\G''})\vert)$ and attaching by a homeomorphism the $(n-1)$-spheres $\vert\textrm{Link}(v',\C'_{\G'})\vert$ and $\vert\textrm{Link}(v'',\C'_{\G''})\vert$.

(i) In this case $\vert\textrm{Star}(v',\C'_{\G'})\vert\cap N(\S_{\G'})=\vert\textrm{Star}(v'',\C'_{\G''})\vert\cap N(\S_{\G''})=\emptyset$, and therefore $M_{\G}=M_{\G'}\#M_{\G''}.$

(ii) In this case $\vert\textrm{Star}(v',\C'_{\G'})\vert\cap N(\S_{\G'})\neq\emptyset\neq\vert\textrm{Star}(v'',\C'_{\G''})\vert\cap N(\S_{\G''})$ and both $\vert\textrm{Star}(v',\C'_{\G'})\vert\cap\vert C(\S'_{\G'},\C''_{\G'})\vert$ and $\vert\textrm{Star}(v'',\C'_{\G''})\vert\cap\vert C(\S'_{\G''},\C''_{\G''})\vert$ are $n$-balls. In fact, they are homeomorphic to $s_n-\textrm{int}(|\textrm{Star}(P,\bar s''_n)|)$, where $P$ is any vertex of the standard $n$-simplex $s_n$.
Moreover, $B'=\vert\textrm{Star}(v',\C'_{\G'})\vert\cap \partial (M_{\G'})$ and $B''=\vert\textrm{Star}(v'',\C'_{\G''})\vert\cap \partial (M_{\G''})$ are both $(n-1)$-balls, since they are homeomorphic to $|\textrm{Link}(P,\bar s''_n)|$, as well as $A'=\vert\textrm{Link}(v',\C'_{\G'})\vert\cap \vert C(\S'_{\G'},\C''_{\G'})\vert$ and $A''=\vert\textrm{Link}(v'',\C'_{\G''})\vert\cap \vert C(\S'_{\G''},\C''_{\G''})\vert$, since they are both the complement of an $(n-1)$-ball in an $(n-1)$-sphere. 
If $M'=M_{\G'}-\textrm{int}(|\textrm{Star}(v',\C'_{\G'})|)$ and $M''=M_{\G''}-\textrm{int}(|\textrm{Star}(v'',\C'_{\G''})|)$, then $M'$ (resp. $M''$) is homeomorphic to $M_{\G'}$ (resp. to $M_{\G''}$) and $A'\subset \partial M'$ (resp. $A''\subset \partial M''$).
Since $M_{\G}$ is obtained by attaching $M'$ with $M''$ via a homeomorphism from $A'$ to $A''$, the proof is achieved.
\end{proof}

\section{\hskip -0.7cm . Results in dimension four}\label{Section8}

The dimension four is the smallest one where quasi-manifolds which are not singular manifolds appear. In this context we have the following characterization.\footnote{Compare Lemma~21 of \cite{[CCDG]}.}

\begin{lemma} \label{prova} Let $\G$ be a $5$-colored graph of order $2p$. Then: 
\begin{itemize}
\item[•] (i) $2|\R_3(\G)|- 3|\R_2(\G)|+10p\ge 0$;
\item[•] (ii) $\widehat M_{\G}$ is a singular manifold if and only if $2|\R_3(\G)|=3|\R_2(\G)|-10p$.
\end{itemize}
\end{lemma}

\begin{proof} For each 3-residue $\L$ of $\G$ we have the relation $2-2\rho_{\L}=b-v/2,$ where $v$ (resp. $b$) is the number of vertices (resp. of bigons) of $\L$, and $\rho_{\L}\ge 0$ denotes the genus (resp. half of the genus) of the orientable (resp. non-orientable) surface $\widehat M_{\L}$. 
%Note that $\rho_{\L}=0$ if and only if $\L$ is ordinary. 
By summing over all 3-residues of $\G$ we obtain
$2|\R_3(\G)|-2\sum_{\L\in\R_3(\G)}\rho_{\L}=3|\R_2(\G)|-10p.$ So we have 
$0\le 2\sum_{\L\in\R_3(\G)}\rho_{\L}=2|\R_3(\G)|-3|\R_2(\G)|+10p$, which proves (i). Since all 3-residues of a singular 4-manifold are ordinary, we obtain (ii).
\end{proof} 

\medskip

Various classification results of 3-manifolds with boundary representable by $4$-colored graphs of small order are contained in \cite{[CM]}, \cite{[CFMT]} and \cite{[CFMT2]}. 

In the 4-dimensional case, as pointed out in Section~\ref{section_suspension}, the 5-colored graph of order two represents $S^4$ and a 5-colored graph of order four represents $S^4$ if it is bipartite and $RP^2\times B^2$ if it is non-bipartite. For order six 5-colored bipartite graphs we have the following result.

\begin{prop} \label{order_6} Let $\G$ be a $5$-colored bipartite graph of order six. Then $M_{\G}$ is one of the following 4-manifolds: $S^4$, $B^4$, $S^1\times B^3$, $S^1\times S^1\times B^2$.
\end{prop}

\begin{proof} First of all, it is well-known that the only closed orientable 4-manifold admitting a representation with six vertices is $S^4$ (see \cite{[CC]}). 
Moreover, the compact orientable $3$-manifolds admitting an order six representation are $S^3$, $S^1\times B^2$ and $S^1\times S^1\times I$ (see \cite{[Li]} and \cite{[CM]}). From Theorem~\ref{suspension} it follows that the three $4$-manifolds $S^4$, $S^1\times B^3$, $S^1\times S^1\times B^2$ admit order six representation which are the suspension of the $4$-colored graphs of order six representing the above 3-manifolds and no other compact $4$-manifold can be obtained via suspension process from a $4$-colored graph. So we can restrict our attention to order six bipartite 5-colored graphs which are not the suspension of a $4$-colored graph. It is easy to see that, up to isomorphism, there is only a graph of such type (depicted on the right of Figure~3), which represents a simply connected $4$-manifold $M$ (by Corollary~\ref{groupM-cor}) with connected spherical boundary.\footnote{From Proposition~\ref{boundary1} the boundary of $M$ results to be a 3-manifold admitting a genus one Heegaard splitting and simple Van Kampen type arguments show that it is simply connected.} The represented manifold turns out to be a 4-ball by simple homology arguments. In fact, by Lemma~\ref{chi} we have $\chi(M)=1$ and, if $M'$ is the closed $4$-manifold obtained from $M$ by capping off its boundary by a 4-ball, then $\chi (M')=2$. Since $M'$ is simply connected and orientable we have $\beta_1(M')=\beta_3(M')=0$ and $\beta_0(M')=\beta_4(M')=1$, therefore $\beta_2(M')=0$ and $M'$ is a $4$-sphere. As a consequence, $M$ is homeomorphic to $B^4$.
\end{proof} 

\medskip

The previous result shows that in dimension $>3$ the boundary of $M_{\G}$ may have spherical components. Order six graphs representing $S^1\times S^1\times B^2$, $S^1\times B^3$ and $B^4$ are depicted in Figures~2 and 3.

\begin{figure}[h!]                      
\begin{center}                         
\includegraphics[width=15cm]{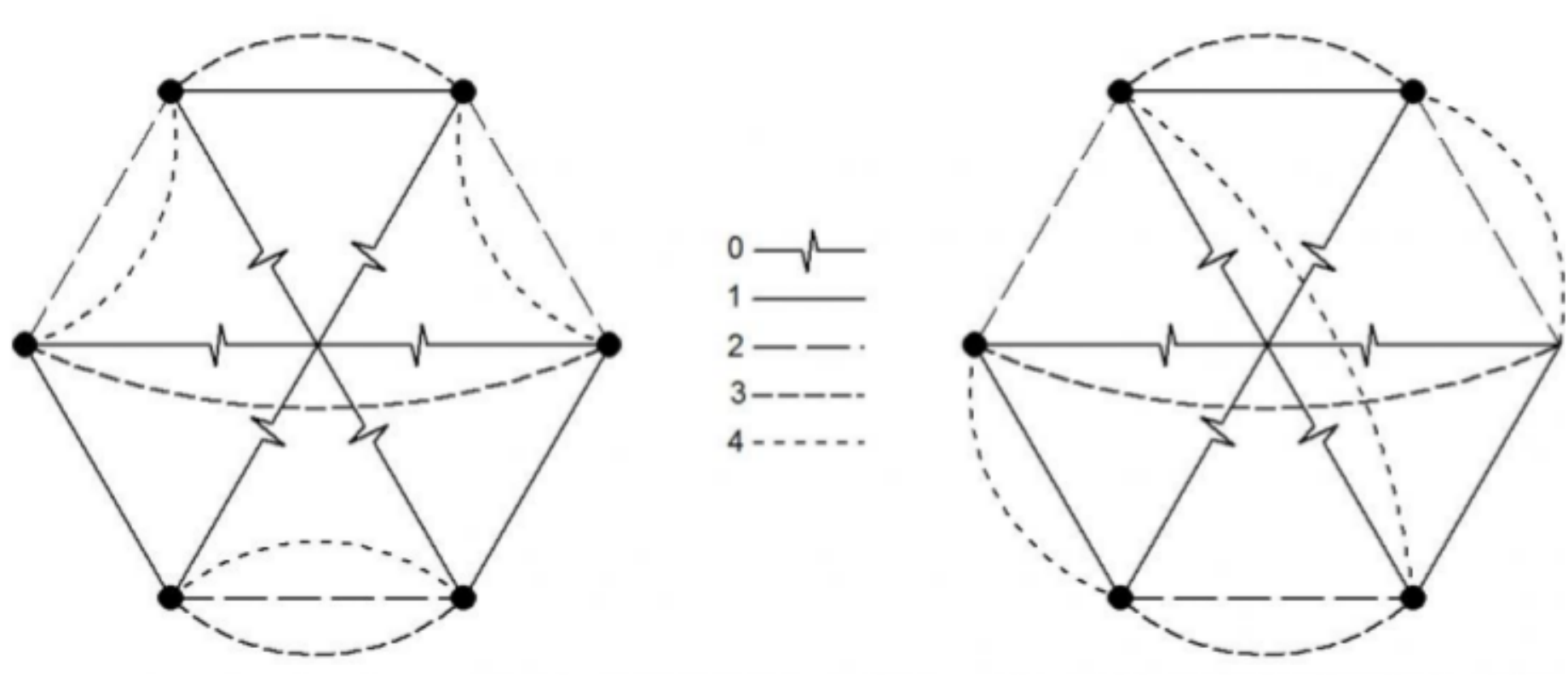}
\caption[legenda elenco figure]{Graphs representing $S^1\times B^3$ and $B^4$, respectively.}\label{fig2}
\end{center}
\end{figure}

\subsection{\hskip -0.7cm . G-degree of supercontracted $\bf 5$-coloured graphs}\label{G-degree}

As sketched in the introduction, a strong interaction is known to exist between edge-colored graphs, representing quasi-manifolds of arbitrary dimension, and random tensor models (see for example \cite{[Gu]}, \cite{[GR]} and \cite{[CCDG]}). In this framework, colored graphs naturally arise as Feynman graphs encoding tensor trace invariants. The key tool for this relashionship is the so-called {\it G-degree} $\omega_G(\G)$ of an $(n+1)$-colored graph $\G$, which drives the $1/N$ expansion\footnote{The coefficients $F_{\omega_G}[\{t_B\}]$ of the formal series are generating functions of bipartite $(n+1)$-colored graphs with fixed G-degree $\omega_G$.}
\begin{equation*} \label{1/N expansion}
 \sum_{\omega_G\ge 0}N^{-\frac{2}{(n-1)!}\omega_G}F_{\omega_G}[\{t_B\}],\ \ \ \ \ \ \ \ \ \ (1)
\end{equation*}
within the $n$-tensor product of the complex space $\mathbb{C}^N$.

The $1/N$ expansion of formula (1) describes the role of colored graphs (and of their G-degree $\omega_G$) within colored tensor models theory and explains the importance of looking for catalogues and classification results concerning all $n$-quasi-manifolds represented by $(n+1)$-colored graphs with a given G-degree.\footnote{A parallel tensor models theory, involving {\it real} tensor variables $T\in (\mathbb{R}^N)^{\otimes 
n}$, has been developed, taking into account also non-bipartite colored graphs (see \cite{[Wi]}): this is why both bipartite and non-bipartite colored graphs will be considered in this context.}
The representation theory via $(n+1)$-colored graphs described in the present paper for all $n$-quasi-manifolds (and their associated compact $n$-manifold possibly with boundary) might be a significant tool for this purpose.

As previously cited, several classification results for 3-manifolds with boundary represented by 4-colored graphs has been recently obtained. Nevertheless, dimension four appears to be a very interesting context within this approach, since it is the least dimension in which colored graphs may represent quasi-manifolds with non-isolated singularities. In this direction, Proposition~\ref{order_6} seems to be particularly significant.

The G-degree arises from the existence of particular embeddings of colored graphs into closed surfaces.

\begin{prop} \cite{[G2]} \label{reg_emb}
Let $\G$ be a bipartite (resp. non-bipartite) $(n+1)$-colored graph of order $2p$. Then for each cyclic permutation $\varepsilon = (\varepsilon_0\ \varepsilon_1\ \ldots\ \varepsilon_n)$ of $\Delta_n$, up to inversion, there exists a cellular embedding, called \emph{regular}, of $\G$ into an orientable (resp. non-orientable) closed surface $F_{\varepsilon}(\Gamma)$
whose regions are bounded by the images of the $\{\varepsilon_j,\varepsilon_{j+1}\}$-bigons, for each $j \in \mathbb Z_{n+1}$.
Moreover, the genus (resp. half the genus)  $\rho_{\varepsilon} (\Gamma)$ of $F_{\varepsilon}(\Gamma)$ satisfies
\begin{equation*}
\chi (F_\varepsilon(\Gamma)) =  2 - 2\rho_\varepsilon(\Gamma)= \sum_{j\in \mathbb{Z}_{n+1}} g_{\varepsilon_j,\varepsilon_{j+1}} + (1-n)p. \ \ \ \ \ \ \ \ \ \ (*)
\end{equation*}
No regular embeddings of $\G$ exist into non-orientable (resp. orientable) surfaces.
\end{prop}

The G-degree of a colored graph is defined in terms of these embeddings as follows. Let $\G$ be an $(n+1)$-colored graph, with $n\ge 2$, and let $\Xi_n$ be 
%$\{\varepsilon^{(1)}, \varepsilon^{(2)}, \dots , \varepsilon^{(\frac {n!} 2)}\}$ 
the set of the $n!/2$ cyclic permutations of $\Delta_n$, up to inversion. For any $\varepsilon\in\Xi_n$, the genus $ \rho_{\varepsilon}(\G)$  is called the \emph{regular genus of $\Gamma$ with respect to $\varepsilon$}. Then, the \emph{Gurau degree} (or \emph{G-degree} for short) of $\Gamma$ is defined as

\begin{equation*}
\omega_{G}(\Gamma) \ = \ \sum_{\varepsilon\in\Xi_n} \rho_{\varepsilon}(\Gamma).
\end{equation*}

%and the \emph{regular genus} of $\Gamma$, denoted by $\rho(\Gamma)$, is defined as

%\begin{equation*}
%\rho(\Gamma) \ = \ \min\, \Big\{\rho_{\varepsilon^{(i)}}(\Gamma)\ /\ i=1,\ldots,\frac {d!} 2\Big\}.
%\end{equation*}

%\end{definition}

For $n=2$ any bipartite (resp. non-bipartite) $3$-colored graph $\G$ represents an orientable (resp. non-orientable) surface $\widehat M_{\G}$ and $\omega_G(\Gamma)$ is exactly the genus (resp. half of the genus) of $\widehat M_{\G}$.
On the other hand, for $n\geq 3$, the G-degree of any $(n+1)$-colored graph is proved to be a non-negative {\it integer}, even in the non-bipartite case (see \cite{[CCDG]}).

In dimension 4 the G-degree is always a multiple of 3, giving rise to the {\it reduced} G-degree $\omega'_G(\G)=\omega_G(\G)/3$. Moreover, from Theorems~1 and~2 of \cite{[CG]} we know that when a 5-colored graph is either bipartite or represents a singular 4-manifold, then $\omega'_G(\G)$ is even. This fact implies that: 
\begin{itemize}
\item in the 4-dimensional complex contest, the only non-vanishing terms in the $1/N$ expansion of (1) are the ones corresponding to even powers of $1/N$;
\item in the 4-dimensional real tensor models framework, where also non-bipartite graphs are involved, only 5-colored graphs representing quasi-manifolds which are not singular manifolds may appear in the terms corresponding to even powers of $1/N$.
\end{itemize}

Now we can give a characterization of supercontracted\footnote{Lemma~\ref{no-dipole}(i)  proves that any closed n-manifold M can be represented by a supercontracted (n+1)-colored graph, called a {\it crystallization} of M (see \cite{[FGG]}).} 5-colored graphs with $\omega'_G\le 3$.

In order to do that, define the cyclic permutation $\varepsilon^c$ of $\widehat c$, for each $c\in\Delta_4$, in the following way:
$$\varepsilon^0=(1\ 3\ 4\ 2),\ \varepsilon^1=(0\ 3\ 2\ 4),\ \varepsilon^2=(0\ 3\ 4\ 1),\ \varepsilon^3=(0\ 2\ 1\ 4),\ \varepsilon^4=(0\ 2\ 3\ 1).$$
If $\G$ is a 5-colored graph of order $2p$, then the sum of the relations $(*)$ in Proposition~\ref{reg_emb} over all $\widehat c$-residues of $\G$ gives: 
$$2g_{\widehat c}-2\rho_{\widehat c}=\sum_{i\in Z_4} g_{\varepsilon_i^c,\varepsilon_{i+1}^c} -2p\ \ \ \ (**),$$
where $\rho_{\widehat c}$ denotes the sum of the regular genera of the ${\widehat c}$-residues with respect to $\varepsilon^c$.

By summing the five relations $(**)$ over $c\in \Delta_4$, we obtain
$$2|\R_4(\G)|-2\sum_{c\in \Delta_4}\rho_{\widehat c}=2|\R_2(\G)|-10p.$$
%$$2\sum_{c\in \Delta_4}g_{\widehat c}-2\sum_{c\in \Delta_4}\rho_{\widehat c=2\sum_{r,s\in \Delta_4}g_{r,s}-10p.$$
The sum $\sum_{c\in \Delta_4}\rho_{\widehat c}$ is called the {\it subdegree} $\rho_{G}(\G)$ of the 5-colored graph $\G$. Therefore 
$$\rho_{G}(\G)=|\R_4(\G)|+5p-|\R_2(\G)|.$$ 
%$$\rho_{G}(\G)=\sum_{c\in \Delta_4}g_{\widehat c}+5p-\sum_{r,s\in \Delta_4}g_{r,s}.$$ 
This definition is motivated by the following result.

\begin{lemma} \label{subdegree}
Let $\omega_G(\G_{\widehat c})$ be the sum of the G-degrees of the connected components of $\G_{\widehat c}$, then $$\sum_{c\in \Delta_4}\omega_G(\G_{\widehat c})=3\rho_G(\G).$$
\end{lemma}

\begin{proof} The relations of Proposition~7 and Lemma~13 of \cite{[CCDG]} for $d=4$ give: 
$3(4+6p-|\R_2(\G)|)=\omega_G(\G)=3(4+p-|\R_4(\G)|)+\sum_{c\in \Delta_4}\omega_G(\G_{\widehat c})$ and therefore $\sum_{c\in \Delta_4}\omega_G(\G_{\widehat c})=3(5p-|\R_2(\G)|+|\R_4(\G)|)=3\rho_G(\G)$. 
%$3(4+6p-\sum_{r,s\in \Delta_4}g_{rs})=\omega_G(\G)=3(p+4-\sum_{c\in \Delta_4}g_{\widehat c})+\sum_{c\in \Delta_4}\omega_G(\G_{\widehat c})$ and therefore $\sum_{c\in \Delta_4}\omega_G(\G_{\widehat c})=3(5p-\sum_{r,s\in \Delta_4}g_{rs})+\sum_{c\in \Delta_4}g_{\widehat c})=3\rho_G(\G)$. 
\end{proof}

%\medskip

%It is straightforward that the only supercontracted 5-colored graph $\G$ with $\omega'_G(\G)=0=\omega_G(\G)$ has order two.

\begin{prop} Let $\G$ be a supercontracted 5-colored graph, then: 
\begin{itemize}
\item[•] (i) $\omega'_G(\G)\neq 1$;
\item[•] (ii) $\omega'_G(\G)=0$ if and only if $\G$ is the order two graph (representing $S^4$);
\item[•] (iii) $\omega'_G(\G)=2$ if and only if $\G$ is the order four graph of  Figure~4 (representing $S^4$);
\item[•] (iv) $\omega'_G(\G)=3$ if and only if $\G$ is the order four graph in the left of Figure~5 (representing $RP^2\times B^2$).
\end{itemize}
\end{prop}

\begin{proof} It is straightforward that the order two graph has zero (reduced) G-degree. Let $2p$ be the order of $\G$. Up to isomorphism there is only a bipartite supercontracted graph (resp. two non-bipartite supercontracted graphs) of order four, namely the one of  Figure~4 (resp. the two of  Figure~5), and it has reduced G-degree $=2$ (resp. they have reduced G-degree $=3$ and $=4$ respectively). Furthermore, up to isomorphism there are 8 bipartite (resp. 31 non-bipartite) supercontracted graphs of order six\footnote{This result has been obtained by a computer program by adapting the algorithmic procedure described in \cite{[CC]}.}, and a direct computation says that their reduced G-degree is always $>3$.

From the proof of Lemma~\ref{subdegree} we have: $\rho_G(\G)=\omega_G(\G)/3-4-p+|\R_4(\G)|=\omega'_G(\G)+1-p$, since $\G$ is supercontracted, and therefore $\omega'_G(\G)\ge p-1$. If $\omega'_G(\G)=0$ then $p\le 1$, which proves (ii). 
If $\omega'_G(\G)=1$ then $p\le 2$, which proves (i). 
If $\omega'_G(\G)=2$ then $p\le 3$, which proves (iii). 
If $\omega'_G(\G)=3$ then $p\le 4$. In this case, if $p=4$ then $\sum_{c\in \Delta_4}\omega_G(\G_{\widehat c})=3\rho_G(\G)=0$. This means that all the 4-residues of $\G$ represent $S^3$ (see Propositions~8 and~9 of \cite{[CCDG]}) and therefore $\widehat M_{\G}$ should be a closed 4-manifold. By Theorem~2 of \cite{[CG]} the G-degree $\omega_G(\G)$ would be a multiple of six, in contrast with the assumption $\omega_G(\G)=9$. This concludes the proof.
\end{proof}

\begin{figure}[h!]                      
\begin{center}                         
\includegraphics[width=10cm]{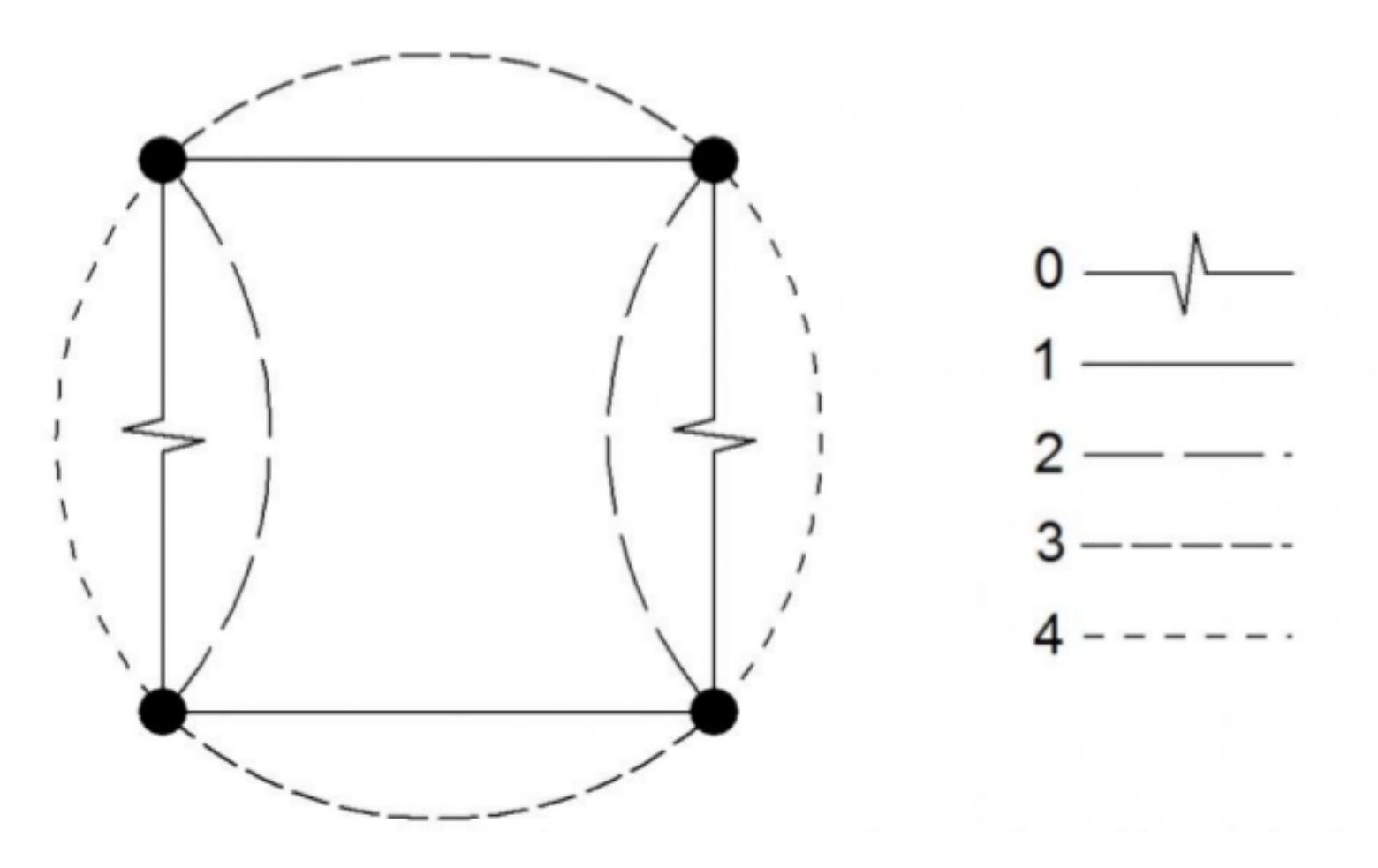}
\caption[legenda elenco figure]{A graph representing $S^4$.}\label{fig3}
\end{center}
\end{figure}

\begin{figure}[h!]                      
\begin{center}                         
\includegraphics[width=15cm]{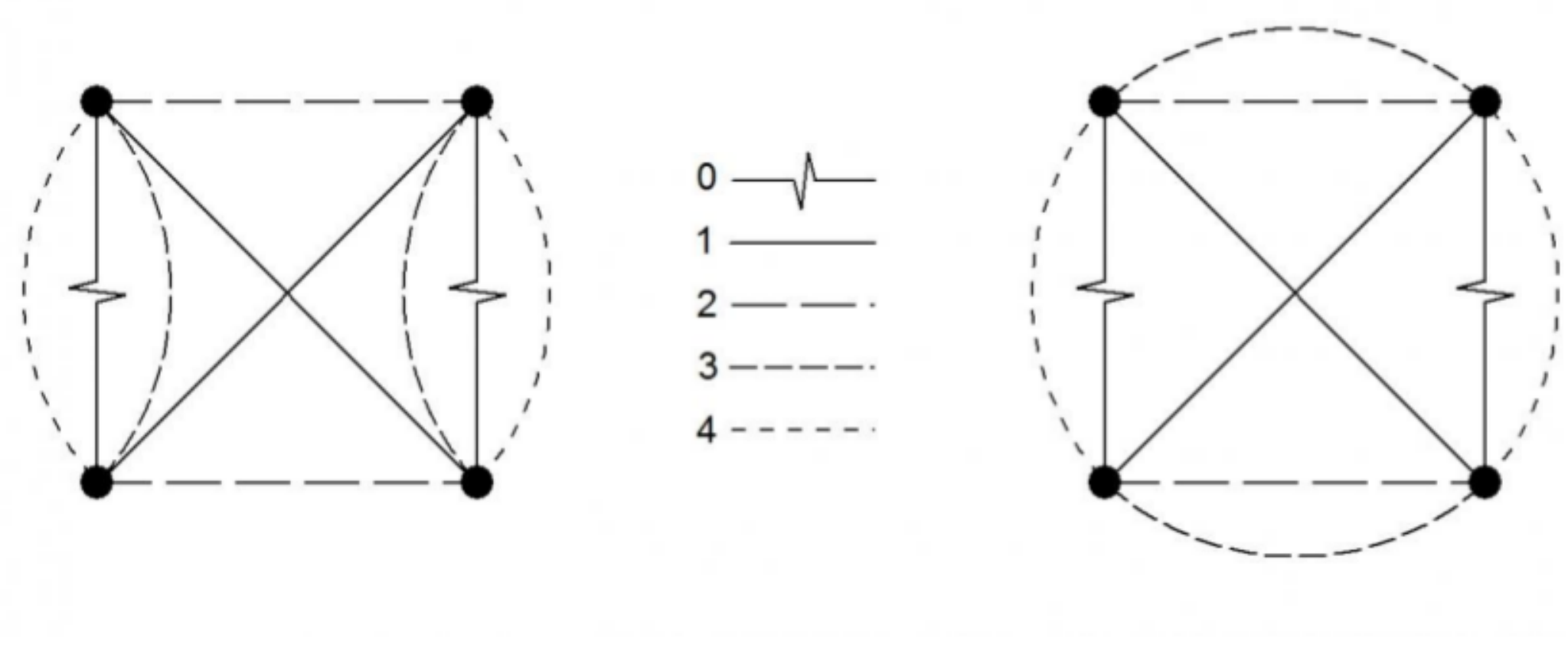}
\caption[legenda elenco figure]{Two graphs representing $RP^2\times B^2$.}\label{fig4}
\end{center}
\end{figure}
\bigskip
\bigskip

{\it Acknowledgement.} 
The authors have been supported by the "National Group for
Algebraic and Geometric Structures, and their Applications" (GNSAGA-INdAM) and
University of Modena and Reggio Emilia and University of Bologna, funds for selected research topics. The authors wish to thank Paola Cristofori for her help in generating the catalogue of all order six 5-colored graphs.

\bigskip

\begin{flushleft}

Luigi~GRASSELLI\\
Dipartimento di Scienze e Metodi dell'Ingegneria, Universit\`a di Modena e Reggio Emilia\\
Via Giovanni Amendola, 41-43, 42122 Reggio Emilia, ITALY\\
e-mail: \texttt{luigi.grasselli@unimore.it}\\

\bigskip

Michele~MULAZZANI\\
Dipartimento di Matematica and ARCES, Universit\`a di Bologna\\
Piazza di Porta San Donato 5, 40126 Bologna, ITALY\\
e-mail: \texttt{michele.mulazzani@unibo.it}\\

\end{flushleft}

\end{document}